\documentclass[reqno,11pt]{amsart}

\usepackage{xcolor}
\usepackage{graphicx}
\usepackage[hidelinks]{hyperref}
\usepackage{amscd}

\usepackage{latexsym,array}
\usepackage{amsfonts}
\usepackage{shadow}
\usepackage{amsbsy}
\usepackage{amssymb, amsmath, enumerate}
\usepackage{epsfig}
\usepackage{dsfont}
\usepackage{mathtools}
\usepackage{doi}
\usepackage[notref, notcite, final]{showkeys}
\usepackage{fullpage}
\usepackage{comment}
\usepackage[latin1]{inputenc}

\usepackage{cleveref}

\numberwithin{equation}{section}

\newcommand{\qcs}{\mathcal Q_{c,s}(T)^{-1}}
\newcommand{\qcp}{\mathcal Q_{c,p}(T)^{-1}}

\newcommand{\PRes}{\mathcal{Q}}

\theoremstyle{theorem}
\newtheorem{theorem}{Theorem}[section]
\newtheorem{lemma}[theorem]{Lemma}
\newtheorem{corollary}[theorem]{Corollary}
\newtheorem{proposition}[theorem]{Proposition}

\theoremstyle{definition}
\newtheorem{remark}[theorem]{{\bf Remark}}
\newtheorem{definition}[theorem]{Definition}
\newtheorem{prob}[theorem]{Problem}

\newcommand{\cc}{\mathbb{C}}

\newcommand{\hh}{\mathbb{H}}

\newcommand{\pp}{\partial}

\newcommand{\unq}{\underline{q}}

\renewcommand{\Re}{\mathrm{Re}}
\renewcommand{\Im}{\mathrm{Im}}

\crefname{enumi}{}{}
\crefname{enumii}{}{}

\title[]{Axially harmonic functions
  and the \\ harmonic functional calculus on the $S$-spectrum}
\author{FABRIZIO COLOMBO, ANTONINO DE MARTINO, STEFANO PINTON, AND IRENE SABADINI }
\date{}

\begin{document}
	
	\maketitle
	
	\begin{abstract}
The spectral theory on the $S$-spectrum was introduced to give
an appropriate mathematical setting to quaternionic quantum mechanics,
but it was soon realized that there were different applications of this theory, for example,
to fractional heat diffusion and to the
spectral theory for the Dirac operator on manifolds.
 In this seminal paper we introduce the  harmonic functional calculus based on the $S$-spectrum and on an integral representation of axially harmonic functions. This calculus can be seen as a bridge between harmonic analysis and the spectral theory.
 The resolvent operator of the harmonic functional calculus
    is the commutative version of the pseudo $S$-resolvent operator. This new calculus also
  appears, in a natural way, in the product rule for the $F$-functional calculus.
\end{abstract}

\medskip
\noindent AMS Classification  47A10, 47A60.


\date{today}
\tableofcontents

\section{Introduction}

The notion of $S$-spectrum and of $S$-resolvent operators for quaternionic linear operators and for linear Clifford operators have been identified only in 2006 using
methods in hypercomplex analysis.
The original motivation for the investigation of a new notion of spectrum
was the paper \cite{BF} of  G. Birkhoff and  J. von Neumann, where the authors showed that quantum mechanics can also be formulated using
quaternions, but they did not
specify what notion of spectrum one should use for quaternionic linear operators.
The appropriate notion is that of $S$-spectrum, which has also been used for the spectral theorem for quaternionic linear operators, see \cite{ack}, and for Clifford algebra linear operators, see \cite{SPECT-CLIFF}.

\medskip
The quaternionic spectral theory on the $S$-spectrum is systematically organised in the books
\cite{FJBOOK,CGKBOOK} while for the Clifford setting  see \cite{MR2752913}.
In particular, the history of the discovery of the $S$-spectrum and the formulation of the $S$-functional calculus are explained
in the introduction of the book \cite{CGKBOOK} with a complete list of references.

\medskip
Nowadays there are several research directions in the area of the spectral theory on the $S$-spectrum, and without claiming completeness we mention:
the characteristic operator function, see \cite{COF},
 slice hyperholomorphic Schur analysis, see \cite{ACSBOOK},
and several applications to fractional powers of vector operators that describe fractional Fourier's laws for nonhomogeneous materials, see for example \cite{frac6,frac4,frac5}.
These results on the fractional powers are based on the
 $H^\infty$-functional calculus (see the seminal papers \cite{Hinfty}, \cite{64FRAC}).

\medskip
{\em The main purpose of this paper is to show that using the Fueter mapping theorem and the spectral theory on the $S$-spectrum we can define a functional calculus for harmonic functions in four variables.
This new calculus can be seen as the harmonic version of the Riesz-Dunford functional calculus.
}
	
Before to explain our results we need some further explanations
of the setting in which we will work.

\subsection{The Fueter-Sce-Qian  extension theorem and spectral theories}
The Fueter-Sce-Qian mapping theorem is a crucial result
 that constructs hyperholomorphic (in a suitable sense) quaternionic or Clifford algebra valued functions starting from  holomorphic functions of one complex variable and it consists of a two steps procedure: the first step gives slice hyperholomorphic functions and
the second one gives the Fueter regular functions in the case of the quaternions, or monogenic functions in the Clifford algebra setting.
Prior to the introduction of slice hyperholomorphic functions,
 the first step
 was simply seen as an intermediate step in the construction.
We point out that the Fueter-Sce-Qian mapping theorem has deep consequences in the spectral theories, in fact
it determines their structures in the hypercomplex setting, as we shall see in the sequel.

\medskip
To further clarify the two steps procedure we summarize the construction in the quaternionic  case by following Fueter, see \cite{F}.
Denoting by $\mathcal{O}(D)$ the set of holomorphic functions on $D$, by ${SH(\Omega_D)}$ the set
of induced functions on $\Omega_D$ (which turn out to be the set of slice hyperholomorphic functions) and by $AM(\Omega_D)$ the set of axially monogenic functions on $\Omega_D$, the Fueter construction can be visualized as:
$$
\begin{CD}
\textcolor{black}{\mathcal{O}(D)}  @>T_{F1}>> \textcolor{black}{SH(\Omega_D)}  @>\ \   T_{F2}=\Delta >>\textcolor{black}{AM(\Omega_D)},
\end{CD}
$$
where $T_{F1}$ denotes the first linear operator of the Fueter construction and $T_{F2}=\Delta$ is the Laplace operator in four dimensions.
The Fueter mapping theorem induces two spectral theories: in the first step we have the spectral theory on the $S$-spectrum associated with the Cauchy formula of slice hyperholomorphic functions;  in the second step we obtain the spectral theory on the monogenic spectrum associated with
the Cauchy formula of monogenic functions.

\medskip
We also note that the Fueter mapping theorem allows to use slice hyperholomorphic functions to obtain the so-called $F$-functional calculus, see \cite{CDS,CG,CS,CSS1}, which is a monogenic functional calculus on the $S$-spectrum.
It is based on the idea of applying the operator $T_{F2}$ to the slice hyperholomorphic Cauchy kernel,
 as illustrated by the diagram:
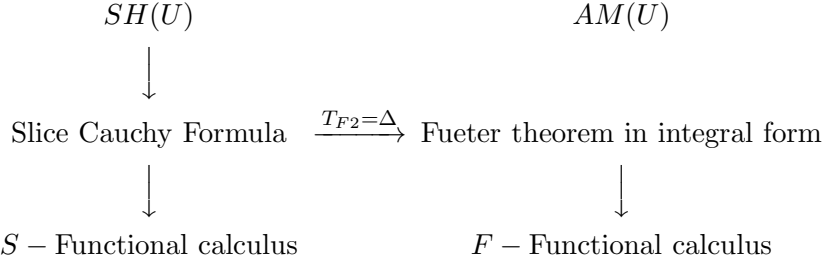
\begin{equation*}
\begin{CD}
{SH(U)} @.  {AM(U)} \\   @V  VV
  @.
\\
{{\rm  Slice\ Cauchy \ Formula}}  @> T_{F2}=\Delta>> {{\rm Fueter\ theorem \ in \  integral\  form}}
\\
@V VV    @V VV
\\
{S-{\rm Functional \ calculus}} @. {F-{\rm Functional \ calculus}}
\end{CD}
\end{equation*}
where $\Delta$ is the Laplace operator in four dimensions.
\begin{remark}
Observe that in the above diagram the arrow from the space of axially monogenic functions $AM(U)$ is missing because the $F$-functional calculus is deduced from the slice hyperholomorphic Cauchy formula. Moreover, we use the set of slice hyperholomorphic functions $SH(U)$, that contains the set of intrinsic functions.
\end{remark}
To proceed further we  fix the notation for the quaternions, that are  defined as follows
	$$
\mathbb{H}=\lbrace{q=q_0+q_1e_1+q_2e_2+q_3e_3 \ \ |\   \ q_0,q_1,q_2,q_3\in\mathbb{R}}\rbrace,
$$
where the imaginary units satisfy the relations
	$$e_1^2=e_2^2=e_3^2=-1\quad \text{and}\quad e_1e_2=-e_2e_1=e_3,\ \  e_2e_3=-e_3e_2=e_1, \ \  e_3e_1=-e_1e_3=e_2.$$
Given $q\in\mathbb H$
we call $ \Re(q):=q_0$ the real part of $q$ and
$ \underline{q}\,=q_1e_1+q_2e_2+q_3e_3$  the imaginary part.
	The modulus of $q\in \mathbb{H}$ is given by
	$|q|=\sqrt{q_0^2+q_1^2+q_2^2+q_3^2},$ the conjugate of $q$ is defined by
	$\overline{q}=q_0- \underline{q}$ and we have $|q|=\sqrt{q\overline{q}}$.
The symbol $ \mathbb{S}$ denotes the unit sphere of purely imaginary quaternions
	$$ \mathbb{S}= \{\underline{q}=q_1e_1+q_2e_2+q_3e_3\ \ | \ \  \, q_1^2+q_2^2+q_3^2=1\} .$$
	Notice that if $J \in \mathbb{S}$, then $J^2=-1$. Therefore $J$ is an imaginary unit, and we denote by
	$$ \mathbb{C}_J=\{u+Jv\ \ \ | \ \ u,v \in \mathbb{R}\},$$
	an isomorphic copy of the complex numbers. Given a non-real quaternion $q= q_0+ \underline{q}= q_0+J_q | \underline{q}|$, we set $J_q= \underline{q}/ | \underline{q}|  \in \mathbb{S}$ and we associate to $q$ the 2-sphere defined by
	$$ [q]:= \{q_0+J |\underline{q}| \ \ | \ \ J \in \mathbb{S}\}.$$	
We recall that the Fueter operator $\mathcal{D}$ and its conjugate $\overline{\mathcal{D}}$ are defined as follows
$$
 \mathcal{D}:= \partial_{q_0}+ \sum_{i=1}^{3} e_i \partial_{q_i}
 \ \ \ \
 {\rm and}\ \ \ \
 \overline{\mathcal{D}}:= \partial_{q_0}- \sum_{i=1}^{3} e_i \partial_{q_i}.
 $$
 The operators $\mathcal{D}$ and $\overline{\mathcal{D}}$ factorize the Laplace operator $\mathcal{D}\overline{\mathcal{D}}=\overline{\mathcal{D}} \mathcal{D}=\Delta$.

\subsection{The fine structure of hyperholomorphic spectral theory and related problems }
In this paper we further refine the above diagram, observing that,
in the case of the quaternions, the map $T_{F2}$ can be factorized as
$T_{F2}=\Delta =\overline{\mathcal{D}}\mathcal{D}$, so there is an intermediate step between slice hyperholomorphic functions and Fueter regular functions, and the intermediate class of functions that appears is the one
 of axially harmonic $AH(U)$ functions, see Definition \ref{axharm}. Thus the diagram becomes as follows:
\begin{equation*}
\begin{CD}
\textcolor{black}{\mathcal{O}(D)}  @>T_{F1}>> \textcolor{black}{SH(\Omega_D)}  @>\ \   \mathcal{D}>>\textcolor{black}{AH(\Omega_D)}
@>\ \   \overline{\mathcal{D}} >>\textcolor{black}{AM(\Omega_D)}.
\end{CD}
\end{equation*}

It is important to define precisely what we mean by intermediate functional calculus
between the $S$-functional calculus and the $F$-functional calculus, both from the points of view of the function theory and of the operator theory. The notions of
{\em fine structure of the spectral theory on the $S$-spectrum} arise naturally from the Fueter extension theorem.

\begin{definition}[Fine structure of the  spectral theory on the $S$-spectrum]
\label{fine}
We will call {\em fine structure of the  spectral theory on the $S$-spectrum}
 the set of functions spaces and the associated functional calculi
 induced by a factorization of the operator $T_{F2}$, in the Fueter extension theorem.
 \end{definition}
 \begin{remark}
 In the Clifford algebra setting the map $T_{F2}$ becomes the Fueter-Sce operator given by  $T_{FS2}= \Delta_{n+1}^{\frac{n-1}{2}}$ and its splitting is more involved.
   We are investigating it in general, when $n$ is odd, and in the case $n=5$ we have a complete description of all the possible fine structures, see \cite{FIVEDIM}. When $n$ is even the Laplace operator has a fractional power and so one has to work in the space of distributions using the Fourier multipliers, see \cite{Q}.
\end{remark}

The fine structure of the quaternionic spectral theory on the $S$-spectrum is illustrated in the following diagram
	{\small
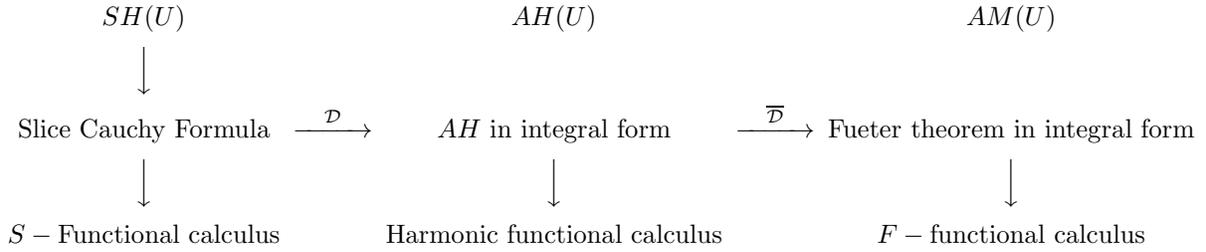
		\begin{equation*}
			\begin{CD}
				{SH(U)} @. {AH(U)}  @.  {AM(U)} \\   @V  VV
				@.
				\\
				{{\rm  Slice\ Cauchy \ Formula}}  @> \mathcal{D} >> {AH {\rm \ in \  integral\  form}}@> \overline{\mathcal{D}} >> {{\rm Fueter\ theorem \ in \  integral\  form}}
				\\
				@V VV    @V VV  @V VV
				\\
				S-{{\rm Functional \ calculus}} @. {{\rm Harmonic \ functional \ calculus}}@. F-{{\rm functional \ calculus}}
			\end{CD}
		\end{equation*}
	}
\newline
\newline
where the description of the central part of the diagram, i.e., the fine structure, is the main topic of this paper.
	\begin{remark}
		As for the  space of axially monogenic functions, the arrow from the space of axially harmonic functions is missing. In fact, like the $F$-functional calculus, also the harmonic functional calculus is deduced from the slice hyperholomorphic Cauchy formula.
\end{remark}
To sum up, the main problems addressed in this paper are:
		
	\begin{prob}\label{p02}
In the Fueter extension theorem
		\begin{equation*}
\begin{CD}
 \textcolor{black}{SH(U)}  @>\ \   \mathcal{D}>>\textcolor{black}{X(U)}
@>\ \   \overline{\mathcal{D}} >>\textcolor{black}{AM(\Omega_D)},
\end{CD}
\end{equation*}
give an integral representation of the functions in the space $X(U):=\mathcal{D}(SH(U))$ and, using this integral transform, define its functional calculus.
	\end{prob}
	\begin{prob}\label{prob2}
		Determine a product rule formula for the $F$- functional calculus.
	\end{prob}
As we will see in the sequel the above problems are related. In fact, the product rule of the $F$-functional calculus is based on the functional calculus in Problem \ref{p02}.

\subsection{Structure of the paper and  main results.}
The paper consists of 9 sections, the first one being this introduction.
In Sections 2 and 3 we give the preliminary material on spectral theories in the hyperholomorphic
setting and the underlying function theories. In Section 3 and 4
we consider axially harmonic functions, see Definition \ref{axharm}.
Using the Cauchy formulas of left (resp. right) slice hyperholomorphic functions we write an integral representation for axially harmonic functions, see
Theorem \ref{qthe}.
More precisely, let $W \subset \mathbb{H}$ be an open set and
let $U$ be a slice Cauchy domain such that $\overline{U} \subset W$. Then for $J \in \mathbb{S}$ and $ds_J=ds(-J)$ we have that
if $f$ is left slice hyperholomorphic on $W$, then the function $ \tilde{f}(q)=\mathcal{D} f(q)$ is harmonic and it admits the following integral representation
			\begin{equation}
				\tilde{f}(q)=- \frac{1}{\pi} \int_{\partial(U \cap \mathbb{C}_J)} \mathcal{Q}_{c,s}(q)^{-1}ds_J f(s),\ \ \ \ q\in  U,
			\end{equation}
where
$$
\PRes_{c,s}(q)^{-1}:=(s^2-2\Re(q)s+|q|^2)^{-1}.
$$
A similar representation is obtained also for right slice hyperholomorphic functions. The integral depends neither on $U$ nor on the imaginary unit $J  \in \mathbb{S}$. Such integral representation is the crucial point to define the harmonic functional calculus,
also called {\em $Q$-functional calculus} because
its resolvent operator is the commutative pseudo $S$-resolvent operator $\mathcal{Q_{c,s}}(T)^{-1}$.

\medskip
Let $T=T_0+T_1e_1+T_2e_2+T_3e_3$ be a quaternionic bounded linear operator with commuting components $T_\ell$, $\ell=0,...,3$, we define the $S$-spectrum of $T$ as
$$
\sigma_S(T)=\{s\in \mathbb{H} \ \ | \ \ s^2-s(T+\overline{T})+T \overline{T} \ \ {\rm is\ not \ invertible}\},
$$
where  $\overline{T}=T_0-T_1e_1-T_2e_2-T_3e_3$.
The commutative pseudo $S$-resolvent operator $\mathcal{Q}_{c,s}(T)^{-1}$ is defined as:
$$
\mathcal{Q}_{c,s}(T)^{-1}=(s^2-s(T+\overline{T})+T \overline{T})^{-1}
$$
for $s\not\in \sigma_S(T)$.
The harmonic functional calculus is defined in Definition \ref{qfun}, but roughly speaking for every function $\tilde{f}=\mathcal{D} f$, with $f$ left slice hyperholomorphic,
  we define the harmonic functional calculus as
		$$
			\tilde{f}(T):= - \frac{1}{\pi} \int_{\partial(U \cap \mathbb{C}_J)} \mathcal{Q}_{c,s}(T)^{-1} ds_J f(s),
		$$
where  $U$ is an arbitrary bounded slice Cauchy domain with $\sigma_{S}(T) \subset U$, $ \overline{U} \subset dom(f)$, $ds_J=ds(-J)$ and $J \in \mathbb{S}$ is an arbitrary imaginary unit.
		A similar definition holds for $\tilde{f}= f\mathcal{D}$ with $f$ right slice hyperholomorphic.

\medskip		
In Section 6 we introduce possible resolvent equations for the harmonic functional calculus
and in Sections 7 and 8 we study some of its properties. In particular, we have
the Riesz projectors, see Theorem \ref{rp}.
		Specifically, let $T=T_1e_1+T_2e_2+T_3e_3$ and assume that the operators $T_\ell$, $\ell=1,\, 2,\, 3$, have real spectrum. Let $\sigma_S(T)=\sigma_1\cup\sigma_2$ with $\operatorname{dist}(\sigma_1,\sigma_2)>0$ and assume that		
		 $G_1,\, G_2\subset\mathbb H$ are two bounded slice Cauchy domains such that $\sigma_1\subset G_1$, $\overline G_1\subset G_2$ and $\operatorname{dist}(G_2,\sigma_2)>0$. Then the operator
		$$
		\tilde P:=\frac 1{2\pi}\int_{\partial (G_2\cap\mathbb C_J)} s\, ds_J\qcs
		$$
		is a projection, i.e.,
		$
		\tilde P^2=\tilde P.
		$
		Moreover, the operator $\tilde{P}$ commutes with $T$.
\\
      Section 9 concludes the paper and contains some properties of the $F$-functional calculus,  such as the product rule,
that can be proved using the $Q$-functional calculus.

	\section{Preliminary results on functions and operators }

We recall some basic results and notations that we will need in the following.
\subsection{Hyperholomorphic functions and the Fueter mapping theorem}

	\begin{definition}
		Let $U \subseteq \mathbb{H}$.
		\begin{itemize}
			\item We say that $U$ is axially symmetric if, for every $u+Iv \in U$, all the elements $u+Jv$ for $J \in \mathbb{S}$ are contained in $U$.
			\item We say that $U$ is a slice domain if $U \cap \mathbb{R} \neq \emptyset$ and if $U \cap \mathbb{C}_J$ is a domain in $\mathbb{C}_J$ for every $J \in \mathbb{S}$.
		\end{itemize}
		
	\end{definition}
	\begin{definition}
		An axially symmetric open set $U \subset \mathbb{H}$ is called slice Cauchy domain if $U \cap \mathbb{C}_J$ is a Cauchy domain in $ \mathbb{C}_J$ for every $J \in \mathbb{S}$. More precisely, $U$ is a slice Cauchy domain if, for every $J \in \mathbb{S}$, the boundary of $ U \cap \mathbb{C}_J$ is the union of a finite number of nonintersecting piecewise continuously differentiable Jordan curves in $ \mathbb{C}_J$.
	\end{definition}
	On axially symmetric open sets we define the class of slice hyperholomorphic functions.
	\begin{definition}[Slice hyperholomorphic functions]
		\label{hyper}
		Let $U \subseteq \mathbb{H}$ be an axially symmetric open set and let
$$
 \mathcal{U} =\{ (u,v)\in \mathbb{R}^2 \ | \ u+\mathbb{S}v \in U\}.
 $$
 We say that a function $f: U \to \mathbb{H}$ of the form
		$$ f(q)=\alpha(u,v)+J\beta(u,v)$$
		is left slice hyperholomorphic if $\alpha$ and $\beta$ are $ \mathbb{H}$-valued differentiable functions such that
\begin{equation}\label{eveodd}
 \alpha(u,v)=\alpha(u,-v), \ \ \ \  \beta(u,v)=-\beta(u,-v) \ \\ \ \  \hbox{for all} \, \, (u,v) \in \mathcal{U},
\end{equation}
		and if $\alpha$ and $\beta$ satisfy the Cauchy-Riemann system
		$$ \partial_u \alpha(u,v)- \partial_v \beta(u,v)=0, \quad \partial_v \alpha(u,v)+ \partial_u \beta(u,v)=0.$$
		Right slice hyperholomorphic functions are of the form
		$$ f(q)= \alpha(u,v)+\beta(u,v)J,$$
		where $\alpha$, $\beta$ satisfy the above conditions.
	\end{definition}
\textbf{Notation}
	The set of left (resp. right) slice hyperholomorphic functions on $U$ is denoted by the symbol $SH_{L}(U)$ (resp. $SH_{R}(U)$). The subset of intrinsic slice hyperholomorphic functions consists of those slice hyperholomorphic functions such that $\alpha$, $\beta$ are real-valued function and is denoted by $ N(U)$.
\begin{remark}\label{remevenodd} If the axially symmetric set $U$ does not intersect the real line then
we can set
$$
 \mathcal{U} =\{ (u,v)\in \mathbb{R}\times\mathbb{R} \ | \ u+\mathbb{S}v \in U\}.
 $$
 and 
 $$ f(q)= \alpha(u,v)+J\beta(u,v),\qquad (u,v)\in\mathcal{U}.$$
 The function $f$ is left slice hyperholomorphic if $\alpha$ and $\beta$ satisfy the Cauchy-Riemann system. Similarly, under the same conditions on $\alpha$ and $\beta$, $ f(q)= \alpha(u,v)+\beta(u,v)J$ is said right slice hyperholomorphic.
 
\end{remark}
	Functions in the kernel of the Fueter operator are called Fueter regular functions and are defined as follows.
	\begin{definition}[Fueter regular functions]\label{d2}
		Let $U\subset \mathbb H$ be an open set. A real differentiable function $ f: U \to \mathbb{H}$ is called left Fueter regular if
		$$\mathcal{D} f(q):=\partial_{q_0} f(q)+\sum_{i=1}^3 e_i \partial_{q_i}f(q)=0.$$
		It is called right Fueter regular if
		$$ f (q) \mathcal{D}:=\partial_{q_0} f(q)+\sum_{i=1}^3  \partial_{q_i}f(q) e_i=0.$$
	\end{definition}

	There are several possible definitions of slice hyperholomorphicity, that are not fully equivalent, but Definition \ref{hyper} of slice hyperholomorphic functions is the most appropriate one for the operator theory and it comes from the Fueter mapping theorem (see \cite{F}).

\begin{theorem}[Fueter mapping theorem]
		\label{F1}
		Let $f_{0}(z)= \alpha(u,v)+i \beta(u,v)$ be a holomorphic function defined in a domain (open and connected) $D$ in the upper-half complex plane and let
		$$ \Omega_D=\{q=q_0+\underline{q} \,\  |\  \ (q_0, |\underline{q}|) \in D\}$$
		be the open set induced by $D$ in $\mathbb{H}$. Then the operator $T_{F1}$ defined by
		$$ f(q)= T_{F1}(f_0):= \alpha(q_0, |\underline{q}|)+ \frac{\underline{q}}{|\underline{q}|}\beta(q_0, |\underline{q}|)$$
		maps the set of holomorphic functions in the set of intrinsic slice hyperholomorphic functions. Moreover, the function
		$$ \breve{f}(q):=T_{F2} \left(\alpha(q_0, |\underline{q}|)+ \frac{\underline{q}}{|\underline{q}|}\beta(q_0, |\underline{q}|)\right),$$
		where $T_{F2}= \Delta$ and $\Delta$ is the Laplacian in four real variables $q_{\ell}$, $ \ell=0,1,2,3$, is in the kernel of the Fueter operator i.e.
		$$ \mathcal{D} \breve{f}=0 \quad \hbox{on} \quad \Omega_D.$$
	\end{theorem}
	
\begin{remark}
	In the late of 1950s, M. Sce extended the Fueter mapping theorem to the Clifford setting
 in the case of odd dimensions, see \cite{S}. In this case, the operator $T_{FS2}$ becomes $T_{FS2}:= \Delta_{n+1}^{\frac{n-1}{2}}$, where $ \Delta_{n+1}$ is the Laplacian in $n+1$ dimensions, so in this case we are dealing with a differential operator. For a translation of M. Sce works in hypercomplex analysis
 with commentaries see \cite{CSS3}; this includes also the version of the Fueter mapping theorem for octonions.
In 1997, T. Qian proved that the Fueter-Sce theorem holds also in the case of even dimensions. In this case the operator $\Delta_{n+1}^{\frac{n-1}{2}}$ is a fractional operator, see \cite{Q, TAOBOOK}.
	\end{remark}
\medskip
Using the theory of monogenic function  A. McIntosh and his collaborators
 introduced the spectral theory on the monogenic spectrum to define functions of
noncommuting operators on Banach spaces. They developed
 the monogenic functional calculus and several of its applications,
 see the books \cite{JM} and \cite{J}.

We now recall the slice hyperholomorphic Cauchy formulas that are the starting point to construct the hyperholomorphic spectral theories on the $S$-spectrum. We will be in need the following result (\cite[Thm. 2.1.22]{CGKBOOK}, \cite[Prop. 2.1.24]{CGKBOOK}).
	
	\begin{theorem}\label{ts}
		Let $s$, $q\in \mathbb H$ with $|q|<|s|$, then
		$$\sum_{n=0}^{+\infty} q^ns^{-n-1}=-(q^2-2\Re(s)q+|s|^2)^{-1}(q-\overline s)$$
		and
		$$\sum_{n=0}^{+\infty} s^{-n-1}q^n=-(q-\overline s)(q^2-2\Re(s)q+|s|^2)^{-1}.$$
		Moreover, for any $s,\, q\in \mathbb H$ with $q\notin [s]$, we have
		$$ -(q^2-2\Re(s)q+|s|^2)^{-1}(q-\overline s)=(s-\overline q)(s^2-2\Re(q)s+|q|^2)^{-1} $$
		and
		$$ -(q-\overline s)(q^2-2\Re(s)q+|s|^2)^{-1}=(s^2-2\Re(q)s+|q|^2)^{-1}(s-\overline q).$$
	\end{theorem}
	In view of Theorem \ref{ts} there are two possible representations of the Cauchy kernels for both the left and the right slice hyperholomorphic functions.
\begin{definition}
Let $s,\, q\in \mathbb H$ with $q\notin [s]$ then we define 
$$ \PRes_{s}(q)^{-1}:=(q^2-2\Re(s)q+|s|^2)^{-1} ,\ \ \ \
\PRes_{c,s}(q)^{-1}:=(s^2-2\Re(q)s+|q|^2)^{-1},
$$
\end{definition}
that are called pseudo Cauchy kernel and commutative pseudo Cauchy kernel, respectively.
\begin{definition}\label{d1}
Let $s,\, q\in \mathbb H$ with $q\notin [s]$ then
\begin{itemize}
			\item We say that the left slice hyperholomorphic Cauchy kernel  $S^{-1}_L(s,q)$ is written in the form I if
			$$S^{-1}_L(s,q):=\PRes_{s}(q)^{-1}(\overline s-q).$$
			\item We say that the right slice hyperholomorphic Cauchy kernel  $S^{-1}_R(s,q)$ is written in the form I if
			$$S^{-1}_R(s,q):= (\overline s-q)\PRes_{s}(q)^{-1}.$$
			\item We say that the left slice hyperholomorphic Cauchy kernel  $S^{-1}_L(s,q)$ is written in the form II if
			$$S^{-1}_L(s,q):=(s-\overline q) \PRes_{c,s}(q)^{-1}.$$
			\item We say that the right slice hyperholomorphic Cauchy kernel  $S^{-1}_R(s,q)$ is written in the form II if
			$$S^{-1}_R(s,q):= \PRes_{c,s}(q)^{-1}(s- \overline q).$$
		\end{itemize}
	\end{definition}
	In this article, unless otherwise specified, we refer to $S^{-1}_L(s,q)$ and $S^{-1}_R(s,q)$  written in the form II.
	
	The following results will be very important in the sequel.
	\begin{lemma} Let  $s\notin [q]$.
		The left slice hyperholomorphic Cauchy kernel $S_L^{-1}(s,q)$ is left slice hyperholomorphic in $q$ and right slice hyperholomorphic in $s$. The right slice hyperholomorphic Cauchy kernel $S_R^{-1}(s,q)$ is left slice hyperholomorphic in $s$ and right slice hyperholomorphic in $q$.
	\end{lemma}

	\begin{theorem}[The Cauchy formulas for slice hyperholomorphic functions]
		\label{Cauchy}
		Let $U\subset\mathbb{H}$ be a bounded slice Cauchy domain, let $J\in\mathbb{S}$ and set  $ds_J=ds (-J)$.
		If $f$ is a left slice hyperholomorphic function on a set that contains $\overline{U}$ then
		\begin{equation}
			f(q)=\frac{1}{2 \pi}\int_{\partial (U\cap \mathbb{C}_J)} S_L^{-1}(s,q)\, ds_J\,  f(s),\qquad\text{for any }\ \  q\in U.
		\end{equation}
		If $f$ is a right slice hyperholomorphic function on a set that contains $\overline{U}$,
		then
		\begin{equation}\label{Cauchyright}
			f(q)=\frac{1}{2 \pi}\int_{\partial (U\cap \mathbb{C}_J)}  f(s)\, ds_J\, S_R^{-1}(s,q),\qquad\text{for any }\ \  q\in U.
		\end{equation}
		These integrals  depend neither on $U$ nor on the imaginary unit $J\in\mathbb{S}$.
	\end{theorem}
	Moreover, for slice hyperholomoprhic functions hold a version of the Cauchy integral theorem
	\begin{theorem}[Cauchy integral Theorem]\label{CIT}
		Let $U\subset \hh$ be open, let $J\in\mathbb S$, and let $f\in SH_L(U)$ and $g\in SH_R(U)$. Moreover, let $D_J\subset U\cap\cc_J$ be an open and bounded subset of the complex plane $\cc_J$ with $\overline D_J\subset U\cap\cc_J$ such that $\partial D_J$ is a finite union of piecewise continuously differentiable Jordan curves. Then
		$$\int_{\partial D_J} g(s)ds_J f(s)=0,$$
		where $ds_J=ds(-J)$.
	\end{theorem}
	Now, we recall what happens when we apply the second Fueter operator $T_{F2}:= \Delta$, where $ \Delta= \sum_{i=0}^3 \partial_{q_i}^2$, to the slice hyperholomorphic Cauchy kernel.
	
	\begin{proposition}\label{Laplacian}
		Let $q$, $s \in \mathbb{H}$ and $q \not\in [s]$. Then:
		\begin{itemize}
			\item
			The function $\Delta S_L^{-1}(s,q)$ is a left  Fueter regular function in the variable $q$ and right slice hyperholomorphic in $s$.
			\item
			The function  $\Delta S_R^{-1}(s,q)$ is a right Fueter regular function in the variable $q$ and left slice hyperholomorphic in $s$.
		\end{itemize}
	\end{proposition}
	In \cite[Thm. 2.2.2]{CGKBOOK} there are the explicit computations of the functions
	$$ (s,q) \mapsto \Delta S_{L}^{-1}(s,q), \qquad (s,q) \mapsto \Delta S_{R}^{-1}(s,q).$$
	
	\begin{theorem}\label{t2}
		Let $q,s\in\mathbb H$ with $q\notin [s]$. Then we have
		$$\Delta S^{-1}_L(s,q)=-4(s-\overline q)(s^2-2\Re(q)s+|q|^2)^{-2}$$
		and
		$$\Delta S^{-1}_R(s,q)=-4(s^2-2\Re(q)s+|q|^2)^{-2}(s-\overline q).$$
	\end{theorem}
	We recall the definition of the $F$-kernels.
	\begin{definition}
\label{FK}
		Let $q$, $s \in \mathbb{H}$. We define for $ s \notin [q]$, the left $F$-kernel as
		$$F_{L}(s,q):= \Delta S^{-1}_L(s,q)=-4(s-\overline q)\PRes_{c,s}(q)^{-2},$$
		and the right $F$-kernel as
		$$F_R(s,q):=\Delta S^{-1}_R(s,q)=-4\PRes_{c,s}(q)^{-2} (s-\overline q).$$
	\end{definition}
	We recall the following relation between the $F$-kernel and the commutative pseudo Cauchy kernel
$\PRes_{c,s}(q)^{-1}$.
	\begin{theorem}\label{t3}
		Let $s,\, q\in\mathbb H$ be such that $q\notin [s]$, then
		$$ F_L(s,q)s-qF_L(s,q)=-4\PRes_{c,s}(q)^{-1} $$
		and
		$$ sF_R(s,q)-F_R(s,q)q=-4\PRes_{c,s}(q)^{-1} $$
	\end{theorem}
	The following result plays a key role, see \cite[Thm. 2.2.6]{CGKBOOK}.
	\begin{theorem}[The Fueter mapping theorem in integral form]
		\label{Fueter}
		Let $U\subset\mathbb{H}$ be a slice Cauchy domain, let $J\in\mathbb{S}$ and set  $ds_J=ds (-J)$.
		\begin{itemize}
			\item
			If $f$ is a left slice hyperholomorphic function on a set $W$, such that $\overline{U} \subset W$, then
			the left Fueter regular function  $\breve{f}(q)=\Delta f(q)$
			admits the integral representation
			\begin{equation}\label{FuetLSEC}
				\breve{f}(q)=\frac{1}{2 \pi}\int_{\partial (U\cap \mathbb{C}_J)} F_L(s,q)ds_J f(s).
			\end{equation}
			\item
			If $f$ is a right slice hyperholomorphic function on a set $W$, such that $\overline{U} \subset W$, then
			the right Fueter regular function $\breve{f}(q)=\Delta f(q)$
			admits the integral representation
			\begin{equation}\label{FuetRSCE}
				\breve{f}(q)=\frac{1}{2 \pi}\int_{\partial (U\cap \mathbb{C}_J)} f(s)ds_J F_R(s,q).
			\end{equation}
		\end{itemize}
		The integrals  depend neither on $U$ and nor on the imaginary unit $J\in\mathbb{S}$.
	\end{theorem}
	
	\subsection{The $S$-functional calculus}
	We now recall some basic facts of the $S$-function calculus, see \cite{CGKBOOK, MR2752913} for more details. Let $X$ be a two sided quaternionic Banach module of the form $X= X_{\mathbb{R}} \otimes \mathbb{H}$, where $X_{\mathbb{R}}$ is a real Banach space. In this paper we consider $\mathcal{B}(X)$ the Banach space of all bounded right linear operators acting on $X$.
	\\In the sequel we will consider bounded operators of the form $T=T_0+T_1e_1+T_2e_2+T_3e_3$, with commuting components $T_{i}$ acting on a real vector space $X_{\mathbb{R}}$, i.e., $T_i \in \mathcal{B}(X_{\mathbb{R}})$ for $i=0,1,2,3$. The subset of $ \mathcal{B}(X)$ given by the operators $T$ with commuting components $T_i$ will be denoted by $ \mathcal{BC}(X)$.

	Now let $T:X \to X$ be a right (or left) linear operator. We give the following.
	\begin{definition}
		Let $T\in \mathcal B(X)$. For $s\in\mathbb H$ we set
		$$ \mathcal Q_s(T):=T^2-2\Re(s)T+|s|^2\mathcal I. $$
		We define the $S$-resolvent set $\rho_S(T)$ of $T$ as
		$$\rho_S(T):=\{s\in\mathbb H: \, \mathcal Q_s(T)^{-1} \in \mathcal B(X)\},$$
		and we define the $S$-spectrum $\sigma_S(T)$ of $T$ as
		$$ \sigma_S(T):=\mathbb H\setminus \rho_S(T). $$
	\end{definition}
	For $s\in\rho_S(T)$, the operator $\mathcal Q_s(T)^{-1}$ is called the pseudo $S$-resolvent operator of $T$ at $s$.

	\begin{theorem}
		Let $T\in\mathcal B(X)$ and $s\in\mathbb H$ with $\|T\|< |s|$. Then we have
		$$ \sum_{n=0}^{\infty}T^n s^{-n-1}=-\mathcal{Q}_s(T)^{-1}(T-\overline s\mathcal I),$$
and
		$$ \sum_{n=0}^{\infty}s^{-n-1}T^n=-(T-\overline s\mathcal I)\mathcal{Q}_s(T)^{-1}. $$
	\end{theorem}
	According to the left or right slice hyperholomorphicity, there exist two different resolvent operators.
	\begin{definition}[$S$-resolvent operators]

	Let $T \in \mathcal{B}(X)$ and $s\in\mathbb \rho_{S}(T)$. Then the left $S$-resolvent operator is defined as
		$$ S^{-1}_L(s,T):=-\mathcal{Q}_s(T)^{-1}(T-\overline s\mathcal I), $$
		and the right $S$-resolvent operator is defined as
		$$ S^{-1}_R(s,T):=-(T-\overline s\mathcal I)\mathcal{Q}_s(T)^{-1}. $$
	\end{definition}
	
The so-called $S$-resolvent equation, see \cite[Thm. 3.8]{ACGS}, involves both $S$-resolvent operators and the Cauchy kernel of slice hyperholomorphic functions and is recalled in the next result:
\begin{theorem}[$S$-resolvent equation]
Let $T \in \mathcal{B}(X)$ then for $s$, $p \in \rho_S(T)$, with  $q \not\in [s]$, we have
	\begin{equation}
		\label{ress}
		S_{R}^{-1}(s,T)S_{L}^{-1}(p,T)=[\left(S_{R}^{-1}(s,T)-S_{L}^{-1}(p,T)\right)p-\bar{s}\left(S_{R}^{-1}(s,T)-S_{L}^{-1}(p,T)\right)] \mathcal{Q}_{s}(p)^{-1},
	\end{equation}
	where $\mathcal{Q}_{s}(p)=p^2-2\Re(s)p+|s|^2$.
	\end{theorem}
	In order to give the definition of the $S$-functional calculus we need the following classes of functions.
\newline
\textbf{Notation}
		Let $T \in \mathcal{B}(X)$. We denote by $SH_L(\sigma_S(T))$, $SH_R(\sigma_S(T))$ and $N(\sigma_S(T))$ the sets of all left, right and intrinsic slice hyperholomorphic functions $f$, respectively, with $ \sigma_S(T) \subset dom(f)$.
	
	\begin{definition}[$S$-functional calculus]
		\label{Sfun}
		Let $T \in \mathcal{B}(X)$. Let $U$ be a slice Cauchy domain that contains $\sigma_S(T)$  and $\overline{U}$ is contained in the domain of $f$.  Set $ds_J=-dsJ$ for $J\in \mathbb{S}$ so we define
		\begin{equation}
			\label{Scalleft}
			f(T):={{1}\over{2\pi }} \int_{\partial (U\cap \mathbb{C}_J)} S_L^{-1} (s,T)\  ds_J \ f(s), \ \ {\rm for\ every}\ \ f\in SH_L(\sigma_S(T))
		\end{equation}
		and
		\begin{equation}
			\label{Scalright}
			f(T):={{1}\over{2\pi }} \int_{\partial (U\cap \mathbb{C}_J)} \  f(s)\ ds_J
			\ S_R^{-1} (s,T),\ \  {\rm for\ every}\ \ f\in SH_R(\sigma_S(T)).
		\end{equation}
	\end{definition}
	The definition of $S$-functional calculus is well posed since the integrals in (\ref{Scalleft}) and (\ref{Scalright}) depend neither on $U$ and nor on the imaginary unit $J\in\mathbb{S}$, see \cite{JFACSS}, \cite[Thm. 3.2.6]{CGKBOOK}.

	\subsection{The $F$-functional calculus}
	Let us consider $T=T_{0}+ T_1e_1+T_2e_2+T_3e_3$ such that $T \in \mathcal{BC}(X)$.
	\begin{definition}\label{QCS}
		Let $T \in \mathcal{BC}(X)$. For $s \in \mathbb{H}$ we set
		$$\mathcal{Q}_{c,s}(T)=s^2-s(T+\overline{T})+T \overline{T},$$
		where $ \overline {T}=T_{0}- T_1e_1-T_2e_2-T_3e_3$. We define the $ F$-resolvent set as
		$$ \rho_F(T)=\{s\in\mathbb H:\, \mathcal Q_{c,s}(T)^{-1}\in\mathcal B(X)\}. $$
		Moreover, we define the $F$-spectrum of $T$ as
		$$ \sigma_{F}(T)= \mathbb{H}\setminus \rho_F(T).$$
	\end{definition}
	By \cite[Prop. 4.14]{CSS1} we have that the $F$-spectrum is the commutative version of the $S$-spectrum, i.e., we have
	$$ \sigma_{F}(T)=\sigma_{S}(T), \qquad T \in \mathcal{BC}(X),$$
and consequently $\rho_F(T)=\rho_S(T)$.

	For $s \in \rho_{S}(T)$ the operator $ \mathcal{Q}_{c,s}(T)^{-1}$ is called the commutative pseudo $S$-resolvent operator of $T$.
	
	It is possible to define a commutative version of the $S$-functional calculus.

	\begin{theorem}
		Let $T \in \mathcal{BC}(X)$ and $s \in \mathbb{H}$ be such that $\| T\| < s$. Then
		$$ \sum_{m=0}^\infty T^m s^{-1-m}=(s \mathcal{I}-\overline T)\mathcal Q_{c,s}(T)^{-1},$$
and
		$$ \sum_{m=0}^\infty s^{-1-m}T^m =\mathcal Q_{c,s}(T)^{-1}(s\mathcal{I}-\overline T).$$
	\end{theorem}
	\begin{definition}
		Let $T \in \mathcal{BC}(X)$ and $s \in \rho_{S}(T)$. We define the left commutative $S$-resolvent operator as
		$$ S^{-1}_L(s,T)=(s \mathcal{I}-\overline T)\mathcal Q_{c,s}(T)^{-1}, $$
		and the right commutative $S$-resolvent operator as
$$
S^{-1}_R(s,T)=\mathcal Q_{c,s}(T)^{-1}(s\mathcal{I}-\overline T).
$$
	\end{definition}

For the sake of simplicity we denote the commutative version of the $S$-resolvent operators with the same symbols used for the noncommutative ones. It is possible to define an $S$-functional calculus as done in Definition \ref{Sfun}.
Below, when dealing with the $S$-resolvent operators, we intend their commutative version.
\newline
\newline
We conclude with the definition of the $F$-functional calculus.	
	\begin{definition}[$F$-resolvent operators] Let $T \in \mathcal{BC}(X)$. We define the left $F$-resolvent operator as
		$$ F_{L}(s,T)=-4(s\mathcal{I}- \overline{T}) \mathcal{Q}_{c,s}(T)^{-2}, \qquad s \in \rho_{S}(T),$$
		and the right $F$-resolvent operator as
		$$ F_{R}(s,T)=-4\mathcal{Q}_{c,s}(T)^{-2}(s\mathcal{I}- \overline{T}) , \qquad s \in \rho_{S}(T).$$
	\end{definition}
	With the above definitions and Theorem \ref{Fueter} at hand, we can recall the $F$-functional calculus was first introduced in \cite{CSS1} and then investigated in \cite{CDS, CG,CS}.
	
	\begin{definition}[The $F$-functional calculus for bounded operators]
Let $U$ be a slice Cauchy domain that contains $\sigma_S(T)$  and $\overline{U}$ is contained in the domain of $f$.
		Let $T= T_1e_1+T_2e_2+T_3e_3 \in\mathcal{BC}(X)$, assume that the operators $T_{\ell}$, $\ell=1,2,3$ have real spectrum and set $ds_J=ds/J$, where $J\in \mathbb{S}$.
		For any function $f\in SH_L(\sigma_S(T))$, we define
		\begin{equation}\label{DefFCLUb}
			\breve{f}(T):=\frac{1}{2\pi}\int_{\pp(U\cap \mathbb{C}_J)} F_L(s,T) \, ds_J\, f(s).
		\end{equation}
		For any $f\in SH_R(\sigma_S(T))$, we define
		\begin{equation}\label{SCalcMON}
			\breve{f}(T):=\frac{1}{2\pi}\int_{\pp(U\cap \mathbb{C}_J)} f(s) \, ds_J\, F_R(s,T).
		\end{equation}
		
	\end{definition}
	The definition of the $F$-functional calculus is well posed since
	the integrals in (\ref{DefFCLUb}) and (\ref{SCalcMON}) depend neither on $U$ and nor on the imaginary unit $J\in\mathbb{S}$.
	
	The left and right $F$-resolvent operators satisfy the equalities in the next result \cite[Thm. 7.3.1]{CGKBOOK}:
	
	\begin{theorem}[Left and right $F$-resolvent equations]\label{fre} Let $T\in\mathcal{BC}(X)$ and let $s\in\rho_S(T)$. The $F$-resolvent operators satisfy the equations
		$$ F_L(s,T)s-TF_L(s,T)=-4\mathcal Q_{c,s}(T)^{-1}$$
		and
		$$ sF_R(s,T)-F_R(s,T)T= -4\mathcal Q_{c,s}(T)^{-1}.$$
	\end{theorem}
	
	\section{Axially harmonic functions}
	In this section we solve the first part of Problem \ref{p02}. We begin by rewriting the Fueter mapping theorem (see Theorem \ref{F1}) in a more refined way, considering the factorization of the Laplace operator $\Delta$ in terms of the Fueter operator $\mathcal{D}$ and its conjugate $\overline{\mathcal{D}}$.
	
	\begin{theorem}[Fueter mapping theorem (refined)]
		\label{F2}
		Let $f_{0}(z)= \alpha(u,v)+i \beta(u,v)$ be a holomorphic function defined in a domain (open and connected) $D$ in the upper-half complex plane and let
		$$ \Omega_D=\{q=q_0+\underline{q} \, | \, (q_0, |\underline{q}|) \in D\}$$
		be the open set induced by $D$ in $\mathbb{H}$. The operator $T_{F1}$ defined by
		$$ f(q)= T_{F1}(f_0):= \alpha(q_0, |\underline{q}|)+ \frac{\underline{q}}{|\underline{q}|}\beta(q_0, |\underline{q}|)$$
		maps the set of holomorphic functions in the set of intrinsic slice hyperholomorphic functions. Then the function
		$$ \tilde{f}(q):=T'_{F2} \left(\alpha(q_0, |\underline{q}|)+ \frac{\underline{q}}{|\underline{q}|}\beta(q_0, |\underline{q}|)\right),$$
		where $T'_{F2}:= \mathcal{D}$ is the Fueter operator, is in the kernel of the Laplace operator, i.e.,
		$$ \Delta \tilde{f}=0 \quad \hbox{on} \quad \Omega_D.$$
		Moreover,
		$$ \breve{f}(q):=T''_{F2} \tilde{f},$$
		where $T''_{F2}= \overline{\mathcal{D}}$ and $\overline{\mathcal{D}}=\partial_{q_0}-\sum_{i=1}^3e_i \partial_{q_i}$, is the kernel of the Fueter operator, i.e.,
		$$ \mathcal{D} \breve{f}=0 \quad \hbox{on} \quad \Omega_D.$$
	\end{theorem}
\begin{remark}
	The consideration in Remark \ref{remevenodd} holds obviously also in the case of
Theorem \ref{F2}.
\end{remark}
		In Theorem \ref{F2} we have applied to the slice hyperholomorphic function $f$ firstly the Fueter operator and then the operator $\overline{\mathcal{D}}$, while in Theorem \ref{F1} we apply directly the Laplacian.
Therefore, there is a class of functions that lies between the class of slice hyperholomorphic functions and the class of axially monogenic functions: it is the so-called class of axially harmonic functions that we introduce below.

\begin{definition}[Axially harmonic function]\label{axharm}
Let $U \subseteq \mathbb{H}$ be an axially symmetric open set not intersecting the real line, and let
$$
 \mathcal{U} =\{ (u,v)\in \mathbb{R}\times\mathbb{R}^+ \ | \ u+\mathbb{S}v \in U\}.
 $$
 Let $f: U \to \mathbb{H}$ be a function, of class $\mathcal{C}^3$, of the form
		$$ f(q)=\alpha(u,v)+J\beta(u,v), \quad q=u+Jv, \quad J \in \mathbb{S},$$
		where $\alpha$ and $\beta$ are $ \mathbb{H}$-valued  functions.
More in general, let $U \subseteq \mathbb{H}$ be an axially symmetric open set and let
$$
 \mathcal{U} =\{ (u,v)\in \mathbb{R}^2 \ | \ u+\mathbb{S}v \in U\},
 $$
 and assume that
\begin{equation}\label{eveoddll}
 \alpha(u,v)=\alpha(u,-v), \ \ \ \  \beta(u,v)=-\beta(u,-v) \ \\ \ \  \hbox{for all} \, \, (u,v) \in \mathcal{U}.
\end{equation}
Let us set
$$
\tilde{f}(q):=\mathcal{D}f(q),\ \  {\rm for}\ \  q\in U.
$$
If
$$
\Delta \tilde{f}(q)=0, \ \  {\rm for}\ \  q\in U
$$
we say that $\tilde{f}$ is axially harmonic on $U$.
\end{definition}

The axially monogenic functions satisfy a system of differential equations called Vekua system, see \cite{CS0}. In the case of axially harmonic functions, the functions $A(q_0,r)$ and $B(q_0,r)$ satisfy a second order system of differential equations.

	\begin{theorem}
		Let $U$ be an axially symmetric open set in $ \mathbb{H}$, not intersecting the real line,
and let $ \tilde{f}(q)= A(q_0,r)+ \underline{\omega}B(q_0,r)$ be an axially harmonic function on $U$, $r>0$ and $ \underline{\omega} \in \mathbb{S}$.
Then the functions $A=A(q_0,r)$ and $B=B(q_0,r)$ satisfy the following system
		$$ \begin{cases}
			\partial_{q_0}^2A(q_0,r)+ \partial_r^2A(q_0,r) + \displaystyle{\frac{2}{r} \partial_r A(q_0,r)}=0
\\
			\displaystyle{\partial_{q_0}^2B(q_0,r)+ \partial_r^2 B(q_0,r)+ \frac{2r \partial_r B(q_0,r)-2B(q_0,r)}{r^2}}=0.
		\end{cases}
		$$
	\end{theorem}
	\begin{proof}
		An axially harmonic function is written as
$$
\tilde{f}(q)= A(q_0,r)+ \underline{\omega}B(q_0,r),\ \ \ \ q=q_0+r \underline{\omega} \in U
 $$
		and it is in the kernel of the operator $ \Delta=\mathcal{D}\overline{\mathcal{D}}$.
		We denote the Fueter operator as $\mathcal{D}= \partial_{q_0}+ \partial_{\underline{q}}$ and $\overline{\mathcal{D}}= \partial_{q_0}- \partial_{\underline{q}}$, where $ \partial_{\underline{q}}= e_1 \partial_{q_1}+e_2 \partial_{q_2}+e_3 \partial_{q_3}$.
		We know that (see \cite{Dixan})
		\begin{equation}
			\label{cong}
			\partial_{\underline{q}}(A(q_0,r)+ \underline{\omega}B(q_0,r))= \underline{\omega} \partial_r A(q_0,r)- \partial_r B(q_0,r)- \frac{2}{r} B(q_0,r).
		\end{equation}
		This implies that
\[
\begin{split}
\overline{\mathcal{D}}f(q)&=(\partial_{q_0}-\partial_{\underline{q}})(A(q_0,r)+ \underline{\omega}B(q_0,r))
\\
&
= \left(\partial_{q_0}A(q_0,r)+ \partial_r B(q_0,r)+ \frac{2}{r} B(q_0,r)\right)+ \underline{\omega} \left(\partial_{q_0}B(q_0,r)- \partial_r A(q_0,r)\right).
\end{split}
\]
		By setting
$$
A'(q_0,r):=\partial_{q_0}A(q_0,r)+ \partial_r B(q_0,r)+ \frac{2}{r} B(q_0,r)
\ \ \
{\rm and}
\ \ \
B'(q_0,r):=\partial_{q_0}B(q_0,r)- \partial_r A(q_0,r),
$$
we get
$$
\overline{\mathcal{D}}f(q_0,r)= A'(q_0,r) + \underline{\omega}B'(q_0,r).
$$
Now, by applying another time formula \eqref{cong} we obtain
$$
 \partial_{\underline{q}} \overline{D}f(q)= \underline{\omega} \partial_r A'(q_0,r)- \partial_r B'(q_0,r)- \frac{2}{r} B'(q_0,r).
$$
Therefore we have
\[
\begin{split}
 \Delta f(q)&= \mathcal{D} \overline{\mathcal{D}} f(q)
 \\
 &
 =(\partial_{q_0}+ \partial_{\underline{q}}) \overline{D}f(q_0,r)
 \\
 &
 =\left(\partial_{q_0}A'(q_0,r)- \partial_r B'(q_0,r)- \frac{2}{r} B'(q_0,r)\right)
- \underline{\omega} \left(\partial_{q_0}B'(q_0,r)+ \partial_r A'(q_0,r)\right).
\end{split}
\]
Since the function $f$ is axially harmonic we have $\Delta f(q)=0$, thus we get
		\begin{equation}
			\label{sym}
			\begin{cases}
				\partial_{q_0} A'(q_0,r)- \partial_r B'(q_0,r) - \frac{2}{r} B'(q_0,r)=0\\
				\partial_{q_0}B'(q_0,r)+ \partial_r A'(q_0,r)=0.
			\end{cases}
		\end{equation}
		Now, we write the system \eqref{sym} in terms of $A$ and $B$ by substituting $A'$ and $B'$
		\begin{eqnarray}
			\label{deri1}
			\nonumber
			\partial_{q_0} A'(q_0,r)- \partial_r B'(q_0,r) - \frac{2}{r} B'(q_0,r) &=& \partial^2_{q_0} A(q_0,r)+ \partial_{q_0} \partial_r B(q_0,r)+ \frac{2}{r} \partial_{q_0}B(q_0,r)
\\\nonumber
 & &-\partial_r \partial_{q_0}B(q_0,r)
+ \partial_r^2 A(q_0,r)- \frac{2}{r} \partial_{q_0}B(q_0,r)+ \frac{2}{r} \partial_r A(q_0,r)\\
			&=& \partial^2_{q_0} A(q_0,r)+ \partial_r^2A(q_0,r) + \frac{2}{r} \partial_r A(q_0,r).
		\end{eqnarray}
		\begin{eqnarray}
			\label{deri2}
			\nonumber
			\partial_{q_0}B'(q_0,r)+ \partial_r A'(q_0,r) &=& \partial_{q_0}^2 B(q_0,r)- \partial_{q_0} \partial_r A(q_0,r)+ \partial_{r} \partial_{q_0}A(q_0,r)
\\
\nonumber
&&+ \partial_{r}^2B(q_0,r)+ \frac{2}{r} \partial_r B(q_0,r)- \frac{2}{r^2}B(q_0,r)
\\
			&=& \partial_{q_0}^2 B(q_0,r)+ \partial_r^2 B(q_0,r)+ \frac{2r \partial_r B(q_0,r)-2B(q_0,r)}{r^2}.
		\end{eqnarray}
		By putting \eqref{deri1} and \eqref{deri2} in \eqref{sym} we get the statement.
	\end{proof}

\begin{remark}		
		If we suppose that a function $f$ is harmonic over the ball $B_r(p)$ of radius $r$ and center $p$, and continuous in the closure of the ball we can write
		$$ f(q)= \frac{r^2-|q-p|^2}{|\partial B_r(p) | r} \int_{\partial B_r(p) } \frac{f(y)}{|y-q|^4} dy,$$
where $|\partial B_r(p) |$ is the measure of the sphere.
	\end{remark}
	
	\section{Integral representation axially harmonic functions}
In this section we  show how to write an axially haromonic function in integral form. The main advantage of this approach is that it is enough to compute an integral of slice hyperholomorphic functions in order to get an axially harmonic function. The crucial point to get the integral representation is to apply the Fueter operator $ \mathcal{D}$ to the slice hyperholomorphic Cauchy kernels written in second form, see Definition \ref{d1}.

	\begin{theorem}\label{t1}
		Let $s,\, q\in \mathbb H$ be such that $q\notin [s]$ then
		$$\mathcal{D} S^{-1}_L(s,q)=-2\PRes_{c,s}(q)^{-1}$$
		and
		$$S^{-1}_R(s,q) \mathcal{D}=-2\PRes_{c,s}(q)^{-1} .$$
	\end{theorem}
	\begin{proof}
		We prove only the first equality since the second one follows with similar computations. First we apply $\partial_{q_0}$ and $\partial_{q_i}$ for $i=1,\,2,\,3$ to the left slice hyperholomorphic Cauchy kernel
$$ S^{-1}_L(s,q)=(s-\bar{q}) \mathcal{Q}_{c,s}(q)^{-1}.$$
Thus, we have
		\[
		\begin{split}
			\partial_{q_0} S^{-1}_L(s,q)&=-\PRes_{c,s}(q)^{-1}-(s-\overline q)\PRes_{c,s}(q)^{-2}(-2s+2q_0)\\
			&= -\PRes_{c,s}(q)^{-1}-2q_0(s-\overline q)\PRes_{c,s}(q)^{-2}+2(s-\overline q)\PRes_{c,s}(q)^{-2}s\\
			&= -\PRes_{c,s}(q)^{-1}+\frac {q_0}2  F_L(s,q)-\frac 12 F_L(s,q) s,
		\end{split}
		\]
where $F_L(s,q)$ is the left $F$-kernel, see Definition \ref{FK}. Then for any $i=1,\,2,\,3$ we get
		\[
		\begin{split}
			\partial_{q_i} S^{-1}_L(s,q)&=e_i\PRes_{c,s}(q)^{-1}-2q_i(s-\overline q)\PRes_{c,s}(q)^{-2}\\
			&= e_i\PRes_{c,s}(q)^{-1} +\frac 12q_i F_L(s,q).
		\end{split}
		\]
		Thus, by Theorem \ref{t3}, we obtain
		\[
		\begin{split}
			\mathcal{D} S^{-1}_L(s,q)&=\partial_{q_0} S^{-1}_L(s,q)+\sum_{i=1}^3 e_i\partial_{q_i} S^{-1}_L(s,q)\\
			&=-\PRes_{c,s}(q)^{-1}+\frac {q_0}2 F_L(s,q)-\frac 12 F_L(s,q) s -3\PRes_{c,s}(q)^{-1} +\frac{\unq}{2} F_L(s,q)\\
			&=-4\PRes_{c,s}(q)^{-1}-\frac 12\left( F_L(s,q)s-qF_L(s,q)  \right)\\
			&=-2\PRes_{c,s}(q)^{-1}.
		\end{split}
		\]
	\end{proof}
	\begin{remark}
Although the slice hyperholomorphic Cauchy kernel written in form I is more suitable in various cases, like for the definition of $S$-functional calculus, it does not allow easy computations of $\mathcal{D}S^{-1}_L(s,q)$.
	\end{remark}
	
	We recall the following (see \cite[Lemma 1]{B}).
	\begin{lemma}\label{be}
		For all $n\geq 1$ we have
$$
\mathcal{D}q^n=q^n \mathcal{D}=-2\sum_{k=1}^n q^{n-k}\overline q^{k-1}.
$$
	\end{lemma}
	\begin{remark}
		Since
		$$ \overline{\sum_{k=1}^n q^{n-k}\overline q^{k-1}}=\sum_{k=1}^n q^{n-k}\overline q^{k-1} $$
		we deduce that $\mathcal{D}q^n$ is real.
	\end{remark}
	\begin{definition}\label{d3}
		Let $s,q\in\mathbb H$, we define the commutative $ Q$-series as
		$$-2\sum_{m=1}^{+\infty}\sum_{k=1}^m q^{m-k}\overline q^{k-1} s^{-1-m}
		\ \ \ \ \ \ {\rm and}\ \  \ \ \ \ \
		-2\sum_{m=1}^{+\infty} \sum_{k=1}^m s^{-1-m} q^{m-k}\overline q^{k-1}.$$
	\end{definition}
	\begin{remark}
		The two series in Definition \ref{d3} coincide, where they converge, since $\sum_{k=1}^m q^{m-k}\overline q^{k-1}$ is real.
	\end{remark}
	\begin{proposition}\label{p1}
		For $s,\, q\in\mathbb H$ with $|q|< |s|$, the commutative $ Q$-series converges.
	\end{proposition}
	\begin{proof}
		To prove the convergence, it is sufficient to prove the convergence of the modulus of the series, i.e., we consider
		$$ \sum_{m=1}^{+\infty}2m |q|^{m-1}|s|^{-1-m}.$$
		The last series converges by the ratio test. Indeed, since $|q|< |s|$, we have
		$$\lim_{m\to\infty} \frac{(m+1)|q|^m|s|^{-2-m}}{m|q|^{m-1}|s|^{-1-m}}=\lim_{m\to\infty}\frac{m+1}{m}|q||s|^{-1}<1.$$
	\end{proof}
	\begin{lemma}\label{l1}
		For $q,\, s\in\mathbb H$ such that $|q|<|s|$, we have
		$$\PRes_{c,s}(q)^{-1}=\sum_{m=1}^{+\infty}\sum_{k=1}^m q^{m-k}\overline q^{k-1} s^{-1-m}=\sum_{m=1}^{+\infty} \sum_{k=1}^m s^{-1-m} q^{m-k}\overline q^{k-1}.$$
	\end{lemma}
	\begin{proof}
		We prove the first equality since the second one can be proved in a similar way. By Theorem \ref{ts}, we can expand the left Cauchy kernel as
		$$S^{-1}_L(s,q)=\sum_{m=0}^\infty q^m s^{-1-m}.$$
		By Theorem \ref{t1} and Proposition \ref{p1}, which allows to exchange the series with the Fueter operator, we have
		$$ -2\PRes_{c,s}(q)^{-1}=\mathcal{D} S_{L}^{-1}(s,q)=\sum_{m=0}^\infty \left( \mathcal{D} q^m \right)s^{-1-m}.$$
		We get the statement by applying Lemma \ref{be}.
	\end{proof}
	\begin{remark}
		Using the well-known equality
		$$(a^n-b^n)=(a-b)\sum_{k=1}^na^{n-k}b^{k-1}$$
		for $a=q$ and $b=\overline q$, and by Lemma \ref{be} we have
		$$
		\mathcal{D}q^n=\begin{cases} -2nq^{n-1} & \textrm{if $\Im(q)=0$},\\ -(\unq)^{-1}(q^n-\overline q^n) & \textrm{if $\Im(q)\ne 0$}. \end{cases}
		$$
		With this result, we can prove Theorem \ref{t1} by using the series expansion of the kernel in the following way: if $|q|< |s|$ and $\unq\neq 0$ then
		\[
		\begin{split}
			\mathcal{D} S_{L}^{-1}(s,q)&=\sum_{m=0}^\infty \left( \mathcal{D} q^m \right)s^{-1-m}
\\
&
=-(\unq)^{-1} \big(\sum_{m=1}^{\infty}q^m s^{-1-m}-\sum_{m=1}^{\infty}\overline q^ms^{-1-m} \big)\\
			&=-(\unq)^{-1}(S^{-1}_L(s,q)-S^{-1}_L(s,\overline q))
\\
&
=-(\unq)^{-1}(2\unq \mathcal Q_{c,s}(q)^{-1})
\\
&=-2\mathcal Q_{c,s}(q)^{-1},
		\end{split}
		\]
		if $|q|< |s|$ and $\unq=0$, we have
		\[
		\begin{split}
			\mathcal Q_{c,s}(q) \mathcal{D}S^{-1}_L(s,q)&=(s^2-2qs+q^2)\left(-2\sum_{m=1}^{\infty}mq^{m-1}s^{-1-m}\right)\\
			&=-2\sum_{m=1}^{\infty}mq^{m-1}s^{1-m}+4\sum_{m=1}^\infty mq^ms^{-m}-2\sum_{m=1}^\infty mq^{m+1}s^{-m-1}\\
			&=-2 \sum_{m=0}^\infty q^{m}s^{-m}+2 \sum_{m=1}^\infty m q^m s^{-m}-2\sum_{m=2}^\infty mq^{m} s^{-m}+2\sum_{m=2}^\infty q^{m} s^{-m}\\
			&=-2.
		\end{split}
		\]
	\end{remark}
Now, we study the regularity of the function $\mathcal{D}S^{-1}_L(s,q)$ in both variables.
	\begin{proposition}
		\label{ssm}
		Let $s$, $q \in \mathbb{H}$ be such that $ q \notin [s]$. The function $\mathcal{D}S^{-1}_L(s,q)$ is an intrinsic slice hyperholomorphic function in $s$.
	\end{proposition}
	\begin{proof}
		This follows by Theorem \ref{t1} and the shape of the commutative pseudo Cauchy kernel.
	\end{proof}
	
	\begin{remark}
		\label{noq}
		The function $\mathcal{D}(S^{-1}_L(s,q))$ is not left slice hyperholomorphic in the variable $q$. Indeed, let $q=u+Iv$ for an arbitrary $I\in \mathbb S$ then $\mathcal{Q}_{c,s}(q)^{-1}=(s^2-2us+u^2+v^2)^{-1}$ and we have the following two relations
		\[
		\frac{\partial}{\partial u} \mathcal Q_{c,s}(u+Iv)^{-1}=-(-2s+2u)\mathcal Q_{c,s}(u+Iv)^{-2}
		\]
		and
		\[
		\frac{\partial}{\partial v} \mathcal Q_{c,s}(u+Iv)^{-1}=-2v\mathcal Q_{c,s}(u+Iv)^{-2},
		\]
		which yield
		\[
		\begin{split}
			\left(\frac{\partial}{\partial u}+I\frac{\partial}{\partial v}\right)\mathcal Q_{c,s}(u+Iv)^{-1}&=-(-2s+2u+2Iv)\mathcal Q_{c,s}(u+Iv)^{-2}\\
			&=2(s-q)\mathcal Q_{c,s}(u+Iv)^{-2}=-\frac 12 F_L(s,\overline q) .
		\end{split}
		\]
	\end{remark}
The function $\mathcal{D}S^{-1}_L(s,q)$ turns out to be harmonic in $q$, as proved in the following result.
	
	\begin{proposition}
		\label{harmo}
		Let $s$, $q \in \mathbb{H}$ be such that $ q \notin [s]$. Then the function $\mathcal{D}S^{-1}_L(s,q)$ is harmonic in the real components of $q$.
	\end{proposition}
	\begin{proof}
		The result follows by the facts that the Laplacian is a real operator, thus it commutes with $\mathcal D$, and by Proposition \ref{Laplacian}. Indeed
		$$ \Delta \mathcal{D} S^{-1}_L(s,q)=\mathcal{D} \Delta S^{-1}_L(s,q)= \mathcal{D} F_{L}(s,q)=0.$$
	\end{proof}
	Finally as a consequence of the definition of the $F$-kernel we have:
	\begin{lemma}
\label{anti}
		Let $s$, $q \in \mathbb{H}$ be such that $q \notin [s]$, then
		$$ \mathcal{D}^2 S^{-1}_{L}(s,q)=F_{L}(s, \bar{q}).$$
	\end{lemma}
	\begin{proof}
		By Theorem \ref{t1} we have
		\begin{equation}
			\label{double}
			\mathcal{D}^2 S^{-1}_L(s,q)=-2\mathcal{D} \mathcal{Q}_{c,s}(q)^{-1}.
		\end{equation}
		Firstly, we apply the derivatives with respect to $q_0$ and $q_i$, with $i=1,2,3$ to the commutative pseudo  Cauchy kernel
		$$\frac{\partial}{\partial q_{0}} \mathcal{Q}_{c,s}(q)^{-1}=-2 (-s+q_0)(s^{2}-2q_0s+|q|^2)^{-2},$$
		and for $i=1,2,3$ we get
		$$\frac{\partial}{\partial q_{i}} \mathcal{Q}_{c,s}(q)^{-1}=-2 q_i(s^{2}-2q_0s+|q|^2)^{-2}.$$
		Thus we obtain
		\begin{eqnarray*}
			\mathcal{D} \mathcal{Q}_{c,s}(q)^{-1}&=& \frac{\partial }{\partial q_0}\mathcal{Q}_{c,s}(q)^{-1}+ \sum_{i=1}^3 e_i \frac{\partial }{\partial q_i}\mathcal{Q}_{c,s}(q)^{-1}\\
			&=& -2(-s+q_0+ \underline{q}) (s^2-2q_0s+|q|^2)^{-2}\\
			&=&2 (s-q)(s^2-2q_0s+|q|^2)^{-2}.
		\end{eqnarray*}
		Therefore by \eqref{double} we get
		$$
\mathcal{D}^2 S^{-1}_L(s,q)=-4(s-q)(s^2-2q_0s+|q|^2)^{-2}=F_{L}(s, \bar{q}).
$$
	\end{proof}
\begin{remark}
By Proposition \ref{Laplacian} it is clear that the function $F_L(s, \bar{q})$ is axially anti-monogenic. This observation together with Lemma \ref{anti} imply the construction of the following diagram
\begin{equation}
	\label{anti2}
	\begin{CD}
		\textcolor{black}{\mathcal{O}(D)}  @>T_{F1}>> \textcolor{black}{SH(\Omega_D)}  @>\ \   \mathcal{D}>>\textcolor{black}{AH(\Omega_D)}
		@>\ \   \mathcal{D} >>\textcolor{black}{\overline{AM(\Omega_D)}},
	\end{CD}
\end{equation}
where $\overline{AM(\Omega_D)}$ is the set of axially anti-monogenic functions.
In order to avoid this set of functions in the constructions like the one in \eqref{anti2} we impose that the composition of the operators that lie up the arrows must be the Fueter map (in the case of this paper $ \Delta$). This is very important when we increase the dimension of the algebra, see \cite{FIVEDIM}.

\end{remark}

We observe that if we set $q=u+Iv$ and we apply the 2-dimensional Laplacian
	$$ \Delta_2:= \overline{\partial_I} \partial_I = \left( \frac{\partial}{\partial u}- I \frac{\partial }{\partial v}\right)\left( \frac{\partial}{\partial u}+ I \frac{\partial }{\partial v}\right),$$
to the commutative pseudo Cauchy kernel we get itself elevated to a bigger power.
	\begin{lemma}
		Let $s$, $q=u+Iv \in \mathbb{H}$ be such that $q \notin [s]$, then
		$$ \Delta_2 \mathcal{Q}_{c,s}(q)^{-1}=4 \mathcal{Q}_{c,s}(q)^{-2}.$$
	\end{lemma}
	\begin{proof}
		We set $q=u+Iv$, $I \in \mathbb{S}$. By Remark \ref{noq} we know that
		$$ \partial_I \mathcal{Q}_{c,s}(u+Iv)^{-1}=2(s-u-Iv) \mathcal{Q}_{c,s}(q)^{-2}.$$
		Now, we have
		$$ \frac{\partial}{\partial u}\partial_I \mathcal{Q}_{c,s}(u+Iv)^{-1}=-2 \mathcal{Q}_{c,s}(u+Iv)^{-2}-4(s-u-Iv) \mathcal{Q}_{c,s}(u+Iv)^{-3}(-2s+2u),$$
		and
		$$ \frac{\partial}{\partial v}\partial_I \mathcal{Q}_{c,s}(u+Iv)^{-1}=-2I \mathcal{Q}_{c,s}(u+Iv)^{-2}-8(s-u-Iv) \mathcal{Q}_{c,s}(u+Iv)^{-3} v.$$
		By definition of the $2$-dimensional laplacian and since the variable $s$ commute with $\mathcal Q_{c,s}(u+Iv)$, we get
		\begin{eqnarray*}
			\Delta_2 \mathcal{Q}_{c,s}(q)^{-1}&=& \left(\frac{\partial}{\partial u}-I \frac{\partial}{\partial v} \right)\partial_I \mathcal{Q}_{c,s}(u+Iv)^{-1}\\
			&=& -4(s-u-Iv) \mathcal{Q}_{c,s}(u+Iv)^{-3}(-2s+2u)+8I(s-u-Iv) \mathcal{Q}_{c,s}(u+Iv)^{-3} v\\
			&& -4\mathcal{Q}_{c,s}(u+Iv)^{-2}\\
			&=& 8(s-u-Iv)(s-u)\mathcal{Q}_{c,s}(u+Iv)^{-3}+8I(s-u-Iv)v\mathcal{Q}_{c,s}(u+Iv)^{-2}\\
			&&-4\mathcal{Q}_{c,s}(u+Iv)^{-2}\\
			&=& 8(s^2-su-us+u^2-Isv+Iuv+Isv-Iuv+v^2)\mathcal{Q}_{c,s}(u+Iv)^{-3}\\
			&& -4\mathcal{Q}_{c,s}(u+Iv)^{-2}\\
			&=& 8\mathcal Q_{c,s}(u+Iv)\mathcal{Q}_{c,s}(u+Iv)^{-3}-4\mathcal{Q}_{c,s}(u+Iv)^{-2}\\
			&=& 8 \mathcal{Q}_{c,s}(u+Iv)^{-2}-4\mathcal{Q}_{c,s}(u+Iv)^{-2}=4\mathcal{Q}_{c,s}(u+Iv)^{-2}.
		\end{eqnarray*}
	\end{proof}
We conclude this section with an integral representation of axially harmonic functions that will allow us to define the harmonic functional calculus based on the $S$-spectrum.

	\begin{theorem}[Integral representation of axially harmonic functions]
		\label{qthe}
		Let $W \subset \mathbb{H}$ be an open set. Let $U$ be a slice Cauchy domain such that $\overline{U} \subset W$. Then for $J \in \mathbb{S}$ and $ds_J=ds(-J)$ we have:
		\begin{itemize}
			\item[1)]If $f \in SH_L(W)$, then the function $ \tilde{f}(q)=\mathcal{D} f(q)$ is harmonic and it admits the following integral representation
			\begin{equation}
				\label{qform}
				\tilde{f}(q)=- \frac{1}{\pi} \int_{\partial(U \cap \mathbb{C}_J)} \mathcal{Q}_{c,s}(q)^{-1}ds_J f(s),\ \ \ \ q\in  U.
			\end{equation}
			\item[2)] If $f \in SH_R(W)$, then the function $ \tilde{f}(q)= f(q)\mathcal{D}$ is harmonic and it admits the following integral representation
			\begin{equation}
				\tilde{f}(q)=- \frac{1}{\pi} \int_{\partial(U \cap \mathbb{C}_J)} f(s) ds_J \mathcal{Q}_{c,s}(q)^{-1},\ \ \ \ q\in  U.
			\end{equation}
		\end{itemize}
		The integrals depend neither on $U$ nor on the imaginary unit $J  \in \mathbb{S}$.
	\end{theorem}
	\begin{proof}
		We prove only the first statement because the other proof is similar.
		We can write the function $f$ by using the Cauchy formula for slice hyperholomorphic functions, see Theorem \ref{Cauchy}. Now, by applying the left Fueter operator to $f(q)$ and by Theorem \ref{t1} we get
		$$ \tilde{f}(q)=\mathcal{D} f(q)=\frac{1}{2 \pi} \int_{\partial (U \cap \mathbb{C}_J)} \mathcal{D} S^{-1}_L(s,q)ds_J f(s)=- \frac{1}{\pi} \int_{\partial(U \cap \mathbb{C}_J)} \mathcal{Q}_{c,s}(q)^{-1}ds_J f(s).$$
		Since $\tilde{f}(q)=\mathcal{D} f(q)$ and by Proposition \ref{harmo}, it is immediately verified that $ \tilde{f}(q)$ is a harmonic function. The independence of integral in \eqref{qform} from the set $U$ and the imaginary unit $J \in \mathbb{S}$ follows by the Cauchy formula. 
	\end{proof}
In this section we have described the central part of the following diagram 
\begin{equation}
\label{harm}
	\begin{CD}
		\textcolor{black}{\mathcal{O}(D)}  @>T_{F1}>> \textcolor{black}{SH(\Omega_D)}  @>\ \   \mathcal{D}>>\textcolor{black}{AH(\Omega_D)}
		@>\ \   \overline{\mathcal{D}} >>\textcolor{black}{AM(\Omega_D)}.
	\end{CD}
\end{equation}	
\begin{remark}
In the quaternionic setting it is possible to develop another digram like the one in \eqref{harm}.  This comes from the factorization $ \Delta=\mathcal{D} \overline{\mathcal{D}}$ and is called second fine structure in the quaternionic setting, see Definition \ref{fine}. The set of functions that lies between the set of slice hyperholomorphic functions and the set of axially monogenic functions is the set of axially  polyanalytic functions of order 2, for more details see \cite{DP}.
\end{remark}

	\section{The harmonic functional calculus on the $S$-spectrum}

	In this section we introduce the harmonic functional calculus on the $S$-spectrum, which is based on the integral representation of axially harmonic functions. Recall that $X$ denotes a two-sided quaternionic Banach space.
	
We give meaning to the substitution of the variable $q$ with the operator $T$ in the power series introduced in Definition \ref{d3}.
	\begin{definition}\label{dseries}
		Let $T=T_0+\sum_{i=1}^3e_i T_i\in\mathcal B\mathcal C(X)$, $s\in\mathbb H$, we formally define the commutative pseudo $S$-resolvent series as
		$$ -2\sum_{m=1}^\infty  \sum_{k=1}^m T^{m-k} \overline T^{k-1} s^{-1-m}
		\ \ \ \ \ {\rm and}\ \ \ \ \
		-2\sum_{m=1}^\infty  \sum_{k=1}^m s^{-1-m} T^{m-k} \overline T^{k-1}.$$
	\end{definition}
	\begin{remark}
		The two series in Definition \ref{dseries} coincide, where they converge.
	\end{remark}
	\begin{proposition}\label{p2}
		\label{series}
		Let $T=T_0+\sum_{i=1}^3 e_iT_i\in\mathcal B\mathcal C(X)$, $s\in\mathbb H$ and $\|T\|<|s|$, the series in the Definition \ref{dseries} converges. Moreover, we have
		\begin{equation}\label{es1}
			\sum_{m=1}^\infty  \sum_{k=1}^m T^{m-k} \overline T^{k-1} s^{-1-m}=\sum_{m=1}^\infty  \sum_{k=1}^m s^{-1-m} T^{m-k} \overline T^{k-1}=\mathcal Q_{c,s} (T)^{-1}.
		\end{equation}
	\end{proposition}
	
	\begin{proof}
		For the convergence of the series it is sufficient to prove the convergence of the series of the operator norm:
		\begin{equation}\label{s1}
			\sum_{m=1}^\infty  m\|T\|^{m-1} |s|^{-1-m} .
		\end{equation}
		Since
		$$\lim_{m\to\infty} \frac{(m+1)\|T\|^m|s|^{-2-m}}{m\|T\|^{m-1}|s|^{-1-m}}=\lim_{m\to\infty}\frac{m+1}{m}\|T\||s|^{-1}<1,$$
		by the ratio test the series \eqref{s1} is convergent. To prove the equality \eqref{es1}, we show that
		\begin{equation}\label{es1bis}
			\mathcal Q_{c,s}(T)\left(\sum_{m=1}^\infty  \sum_{k=1}^m T^{m-k} \overline T^{k-1} s^{-1-m}\right)=\left(\sum_{m=1}^\infty  \sum_{k=1}^m T^{m-k} \overline T^{k-1} s^{-1-m}\right)\mathcal Q_{c,s}(T)=\mathcal I .
		\end{equation}
		The first equality in \eqref{es1bis} is a consequence of the following facts: for any positive integer $m$ the operator $\sum_{k=1}^m T^{m-k}\overline T^{k-1}$
does not contain any imaginary units, so it is real and then it commutes with any power of $s$. Secondly, the components of $T$ are commuting among them and the operator $\mathcal Q_{c,s}(T)$, see Definition \ref{QCS}, can be written as: $s^2\mathcal I-2s T_0+\sum_{i=0}^3 T_i^2$.

Now we prove the second equality in \eqref{es1bis}. First we observe that
		\[
		\begin{split}
			&\left(\sum_{m=1}^\infty  \sum_{k=1}^m T^{m-k} \overline T^{k-1} s^{-1-m}\right)\mathcal Q_{c,s}(T)=\left(\sum_{m=1}^\infty  \sum_{k=1}^m T^{m-k} \overline T^{k-1} s^{-1-m}\right)(s^2-s(T+\overline T)+T\overline T)\\
			&= \sum_{m=1}^\infty  \sum_{k=1}^m T^{m-k}\overline{T}^{k-1}s^{1-m}- \sum_{m=1}^\infty  \sum_{k=1}^m T^{m+1-k}\overline T^{k-1}s^{-m} - \sum_{m=1}^\infty \sum_{k=1}^m T^{m-k}\overline T^k s^{-m}+\\
			&\, \, \, \, \, \,+ \sum_{m=1}^{\infty} \sum_{k=1}^m T^{m-k+1} \overline{T}^{k}s^{-1-m}.
		\end{split}
		\]
		Making the change of index $m'=1+m$ in the second and fourth series, we have
		\[
		\begin{split}
			&\left(\sum_{m=1}^\infty  \sum_{k=1}^m T^{m-k} \overline T^{k-1} s^{-1-m}\right)\mathcal Q_{c,s}(T)=\\
			&= \sum_{m=1}^\infty  \sum_{k=1}^m T^{m-k}\overline{T}^{k-1}s^{1-m}- \sum_{m'=2}^\infty  \sum_{k=1}^{m'-1} T^{m'-k}\overline T^{k-1}s^{1-m'} - \sum_{m=1}^\infty \sum_{k=1}^m T^{m-k}\overline T^k s^{-m}\\
			&\, \, \, \, \, \,+\sum_{m'=2}^\infty  \sum_{k=1}^{m'-1} T^{m'-k} \overline T^k s^{-m'}\\
			&=\mathcal{I}+\sum_{m=2}^\infty  \sum_{k=1}^m T^{m-k}\overline{T}^{k-1}s^{1-m}- \sum_{m'=2}^\infty  \sum_{k=1}^{m'-1} T^{m'-k}\overline T^{k-1}s^{1-m'} +\\
			&\, \, \, \, \, \,-\overline{T}s^{-1}- \sum_{m=2}^\infty \sum_{k=1}^m T^{m-k}\overline T^k s^{-m}+\sum_{m'=2}^\infty  \sum_{k=1}^{m'-1} T^{m'-k} \overline T^k s^{-m'}.\\
		\end{split}
		\]
		Simplifying the opposite terms in the first and second series and in the third and fourth series, we finally get
		\[
		\left(\sum_{m=1}^\infty  \sum_{k=1}^m T^{m-k} \overline T^{k-1} s^{-1-m}\right)\mathcal Q_{c,s}(T)=\mathcal I+\sum_{m=2}^\infty\overline T^{m-1}s^{1-m}-\sum_{m=2}^\infty\overline T^{m-1}s^{1-m}=\mathcal I.\\
		\]
	\end{proof}

	\begin{lemma}
		Let $T \in \mathcal B\mathcal C(X)$. The commutative pseudo $S$-resolvent operator $ \mathcal Q_{c,s}(T)^{-1}$ is a $ \mathcal{B}(X)$-valued right and left slice hyperholomorphic function of the variable $s$ in $\rho_S(T)$.
	\end{lemma}
	
	\begin{proof}
		It follows by Proposition \ref{ssm}.
	\end{proof}
\begin{remark}
We point out an important difference between the commutative and the noncommutative pseudo $S$-resolvent operator.
For $T\in \mathcal{B}(X)$ with noncommuting components the operator $\mathcal{Q}_{c,s}(T)$ is not well defined because $T\overline{T}\not=\overline{T}T$. But in the case $T\in \mathcal{BC}(X)$ then
it turns out to be well defined and the inverse is $ \mathcal{B}(X)$-valued slice hyperholomorphic function for $s\in \rho_S(T)$.

The noncommutative pseudo $S$-resolvent operator $\mathcal{Q}_{s}(T)^{-1}$
turns out to be well defined for operators $T\in \mathcal{B}(X)$ with noncommuting components,
but it is not a $ \mathcal{B}(X)$-valued slice hyperholomorphic function.
 \end{remark}	

\begin{remark}
The functional calculus based on axially harmonic functions in integral form will be called
{\em harmonic functional calculus (on the $S$-spectrum)} or, since it is based on the commutative pseudo
$S$-resolvent operator $\PRes_{c,s}(T)^{-1}$, for simplicity it will also be called $Q$-functional calculus.
\end{remark}
	\begin{definition}[Harmonic functional calculus on the $S$-spectrum]
		\label{qfun}
		Let $T \in \mathcal{BC}(X)$ and set $ds_J=ds(-J)$ for $J \in \mathbb{S}$. For every function $\tilde{f}=\mathcal{D} f$ with $f \in SH_{L}(\sigma_S(T))$, we set
		\begin{equation}
			\label{inte1}
			\tilde{f}(T):= - \frac{1}{\pi} \int_{\partial(U \cap \mathbb{C}_J)} \mathcal{Q}_{c,s}(T)^{-1} ds_J f(s),
		\end{equation}
		where $U$ is an arbitrary bounded slice Cauchy domain with $\sigma_{S}(T) \subset U$ and $ \overline{U} \subset dom(f)$ and $J \in \mathbb{S}$ is an arbitrary imaginary unit.
		\\ For every function $\tilde{f}= f\mathcal{D}$ with $f \in SH_{R}(\sigma_S(T))$, we set
		\begin{equation}
			\label{inte2}
			\tilde{f}(T):=- \frac{1}{\pi} \int_{\partial (U \cap \mathbb{C}_J)} f(s) ds_J \mathcal{Q}_{c,s}(T)^{-1},
		\end{equation}
		where $U$ and $J$ are as above.
	\end{definition}
	\begin{theorem}
		The harmonic functional calculus on the $S$-spectrum is well-defined, i.e., the integrals in \eqref{inte1} and \eqref{inte2} depend neither on the imaginary unit $J \in \mathbb{S}$ nor on the slice Cauchy domain $U$.
	\end{theorem}
	\begin{proof}
		Here we show only the case $ \tilde{f}=\mathcal{D}f$ with $f \in SH_L(\sigma_S(T))$, since the other one follows by analogous arguments.
		\\ Since $ \mathcal{Q}_{c,s}(T)^{-1}$ is a right slice hyperholomorphic function in $s$ and $f$ is left slice hyperholomorphic, the independence from the set $U$ follows by the Cauchy integral formula, see Theorem \ref{Cauchy} and Theorem \ref{CIT}.
		\\ Now, we want to show the independence from the imaginary unit.
Let us consider two imaginary units $J$, $I \in \mathbb{S}$ with $J \neq I$ and two bounded slice Cauchy domains $U_q$, $U_s$ with $ \sigma_{s}(T) \subset U_q$, $\overline{U}_q \subset U_s$ and $\overline{U}_s \subset dom(f)$.
 Then every $s \in \partial (U_s \cap \mathbb{C}_J)$ belongs to the unbounded slice Cauchy domain $\mathbb{H}\setminus  U_q $.
		 Recall that $\mathcal Q_{c,q}(T)^{-1}$ is right slice hyperholomorphic on $\rho_S(T)$, also at infinity, since $ \lim_{q \to + \infty} \mathcal Q_{c,q}(T)^{-1}=0$. Thus the Cauchy formula implies
		\begin{eqnarray}
			\nonumber
			\mathcal{Q}_{c,s}(T)^{-1}&=& \frac{1}{2 \pi} \int_{\partial \left((\mathbb{H} \setminus U_q) \cap \mathbb{C}_I \right)} \mathcal Q_{c,q}(T)^{-1} dq_I S^{-1}_R(q,s)\\
			\label{inte3}
			&=& \frac{1}{2 \pi} \int_{\partial (U_q \cap \mathbb{C}_I)} \mathcal Q_{c,q}(T)^{-1} dq_I S^{-1}_L(s,q).
		\end{eqnarray}
		The last equality is due to the fact that $ \partial \left((\mathbb{H} \setminus U_q) \cap \mathbb{C}_I \right)=-\partial (U_q \cap \mathbb{C}_I)$ and $S^{-1}_R(q,s)=-S^{-1}_L(s,q).$ Combining \eqref{inte1} and \eqref{inte3} we get
		\begin{eqnarray*}
			\tilde{f}(T)&=&- \frac{1}{ \pi} \int_{\partial(U_s \cap \mathbb{C}_J)} \mathcal Q_{c,s}(T)^{-1} ds_J f(s)\\
			&=&- \frac{1}{ \pi} \int_{\partial(U_s \cap \mathbb{C}_J)} \left( \frac{1}{2 \pi} \int_{\partial(U_q \cap \mathbb{C}_I)} \mathcal Q_{c,q}(T)^{-1} dq_I S_{L}^{-1}(s,q)\right) ds_J f(s).
		\end{eqnarray*}
		Due to Fubini's theorem we can exchange the order of integration and by the Cauchy formula we obtain
		\begin{eqnarray*}
			\tilde{f}(T)&=& - \frac{1}{ \pi} \int_{\partial (U_q \cap \mathbb{C}_I)} \mathcal Q_{c.q}(T)^{-1}dq_I  \left( \frac{1}{2 \pi} \int_{\partial (U_s \cap \mathbb{C}_J)}   S_{L}^{-1}(s,q) ds_J f(s) \right)\\
			&=& - \frac{1}{\pi} \int_{\partial(U_q \cap \mathbb{C}_I)} \mathcal Q_{c,q}(T)^{-1} dq_I f(q).
		\end{eqnarray*}
		This proves the statement.
	\end{proof}

 \begin{prob}
\label{three}
Let $ \Omega$ be a slice Cauchy domain. It might happen that $f,g\in SH_L(\Omega)$ (resp. $f,g\in SH_R(\Omega)$) and $\mathcal{D}f=\mathcal{D}g$  (resp. $f\mathcal{D}=g\mathcal{D}$). Is it possible to show that for any $T\in\mathcal{BC}(X)$, with $\sigma_S(T)\subset \Omega$, we have $\tilde f(T)=\tilde g(T)$?
 \end{prob}

We start to address the problem by observing that $\mathcal{D}(f-g)=0$ (resp. $(f-g)\mathcal{D}=0$).
Therefore it is necessary to study the set
$$ (\ker{\mathcal{D}})_{SH_L(\Omega)}:=\{f\in SH_L(\Omega): \mathcal{D}(f)=0\}\quad\textrm{resp. $(\ker{\mathcal{D}})_{SH_R(\Omega)}:=\{f\in SH_R(\Omega): (f)\mathcal{D}=0\}$} .$$ 
	\begin{theorem}\label{Tcost} Let $\Omega$ be a connected slice Cauchy domain of $\mathbb H$, then
		\begin{eqnarray*}
			(\ker{\mathcal{D}})_{SH_L(\Omega)}&=&\{f\in SH_L(\Omega): f\equiv \alpha\quad\textrm{for some $\alpha\in\mathbb H$}\}\\
			&=&\{f\in SH_R(\Omega): f\equiv \alpha\quad\textrm{for some $\alpha\in\mathbb H$}\}=(\ker{D})_{SH_R(\Omega)}.
		\end{eqnarray*}
		
	\end{theorem}
	\begin{proof}
		We prove the result in the case $f\in SH_L(\Omega)$ since the case $f\in SH_R(\Omega)$ follows by similar arguments. We proceed by double inclusion. The fact that
$$
(\ker{D})_{SH_L(\Omega)}\supseteq\{f\in SH_L(\Omega): f\equiv \alpha\quad\textrm{for some $\alpha\in\mathbb H$}\}
$$ is obvious. The other inclusion can be proved observing that if $f\in (\ker{\mathcal{D}})_{SH_L(\Omega)}$, after a change of variable if needed, there exists $r>0$ such that the function $f$ can be expanded in a convergent series at the origin
		$$f(q)=\sum_{k=0}^{\infty}q^k \alpha_k\quad\textrm{for $(\alpha_k)_{k\in\mathbb N_0}\subset\mathbb H$ and for any $q\in B_r(0)$}$$
		where $B_r(0)$ is the ball centered at $0$ and of radius $r$. By Lemma \ref{be}, we have
		$$
		0=\mathcal{D}f(q)\equiv\sum_{k=1}^{\infty} \mathcal{D}(q^k)\alpha_k=-2\sum_{k=1}^{\infty}\sum_{s=1}^{k} q^{k-s}\overline q^{s-1} \alpha_k,\quad \forall q\in B_r(0),
		$$
If we restrict the previous series in a neighbourhood  $U$ of $0$ of the real line, we obtain
$$
0= (\mathcal{\mathcal{D}}f)(q_0)= -2\sum_{k=1}^{\infty}k q_0^{k-1}\alpha_k,\quad \forall q_0\in U.
$$
By \cite[Cor. 1.2.6]{KP}, we have
		$$\alpha_k=0,\quad \forall k\geq 1,$$
		which yields $f(q)\equiv \alpha_0$ in $U$ and since $\Omega$ is connected $f(q)\equiv\alpha_0$ for any $q\in\Omega$.
	\end{proof}
We solve the problem \ref{three} in the case in whcih $ \Omega$ is connected.
	\begin{proposition}\label{con}
		Let $T\in\mathcal{BC}(X)$ and let $U$ be a  connected slice   Cauchy domain with $\sigma_S(T)\subset U$. If $f,g\in SH_L(U)$ (resp. $f,g\in SH_R(U)$) satisfy the property $\mathcal{D}f=\mathcal{D}g$ (resp. $f\mathcal{D}=g\mathcal{D}$) then $\tilde f(T)=\tilde g(T)$.
	\end{proposition}
	\begin{proof}
		We prove the theorem in the case $f,g\in SH_L(\Omega)$ since the case $f,g\in SH_R(\Omega)$ follows by similar arguments. By definition of the harmonic functional calculus on the $S$-spectrum, see Definition \ref{qfun}, we have
		$$
		\tilde f(T)-\tilde g(T)=-\frac 1{\pi}\int_{\partial (U\cap\cc_J)} \mathcal Q_{c,s}(T)^{-1}ds_J(f(s)-g(s)).
		$$
		Since $\mathcal Q_{c,s}(T)^{-1}$ is slice hyperholomorphic in the variable $s$ by Theorem \ref{Cauchy}, we can change the domain of integration to $B_r(0)\cap\cc_J$ for some $r>0$ with $ \| T \| <r$. Moreover, by hypothesis we have that $f(s)-g(s) \in (\ker{D})_{SH_L(\Omega)}$, thus by Theorem \ref{Tcost} and Proposition \ref{p2} we get
		\[
		\begin{split}
			\tilde f(T)-\tilde g(T)&=-\frac 1{\pi}\int_{\partial (B_r(0)\cap\cc_J)} \mathcal Q_{c,s}(T)^{-1}ds_J(f(s)-g(s))
\\
&
=\int_{\partial (B_r(0)\cap\cc_J)} \mathcal Q_{c,s}(T)^{-1}ds_J\alpha\\
			&= \sum_{m=1}^\infty  \sum_{k=1}^m T^{m-k} \overline T^{k-1} \int_{\partial (B_r(0)\cap \cc_J)}s^{-1-m} ds_J \alpha=0.
		\end{split}
		\]
	\end{proof}

In order to solve Problem \ref{three}, in the case $ \Omega$ not connected, we need the following lemma, which is based on the monogenic functional calculus developed by McIntosh and collaborators, see \cite{JM,J}.

	\begin{lemma}
		\label{mono}
		Let $T \in \mathcal{BC}(X)$ be such that $T= T_1e_1+T_2 e_2+T_3 e_3$, and assume that the operators $T_{\ell}$, $\ell=1,2,3$, have real spectrum. Let $G$ be a bounded slice Cauchy domain such that $(\partial G) \cap \sigma_{S}(T)= \emptyset$. For every $J \in \mathbb{S}$ we have
		\begin{equation}
			\int_{\partial{(G \cap \mathbb{C}_J)}} \mathcal{Q}_{c, s}(T)^{-1} ds_J=0.
		\end{equation}
	\end{lemma}
	\begin{proof}
		Since $\Delta(1)=0$ and $\Delta(q)=0$, by Theorem \ref{Fueter} we also have
		\begin{equation}\label{nova}
			\int_{\partial{(G \cap \mathbb{C}_J)}} F_L(s,q) ds_J=\Delta(1)=0,
		\end{equation}
		and
		\begin{equation}\label{nova_bis}
		 \int_{\partial{(G \cap \mathbb{C}_J)}} F_L(s,q) ds_Js =\Delta(q)=0,
		\end{equation}
		for all $q \notin \partial G$ and $J \in \mathbb{S}$. By the monogenic functional calculus \cite{JM,J} we have
		$$ F_L(s,T)= \int_{\partial \Omega} G(\omega,T) \mathbf{D}\omega F_L(s,\omega),$$
		where $\mathbf{D}\omega$ is a suitable differential form, the open set $ \Omega$ contains the left spectrum of $T$ and $G(\omega,T)$ is the Fueter resolvent operator. By Theorem \ref{fre} we can write $$\mathcal Q_{c,s}(T)^{-1}=-\frac 14 (F_L(s,T)s-TF_L(s,T)),$$ thus  we have
		\begin{eqnarray*}
			&& \int_{\partial{(G \cap \mathbb{C}_J)}} \mathcal{Q}_{c,s}(T)^{-1} ds_J=-\frac 14 \int_{\partial{(G \cap \mathbb{C}_J)}}  F_L(s,T)s-TF_L(s,T) ds_J\\
			&&=-\frac 14 \left( \int_{\partial{(G \cap \mathbb{C}_J)}}  \int_{\partial\Omega} G(\omega, T)\mathbf{D}\omega F_L(s,\omega)s\, ds_J - T\int_{\partial{(G \cap \mathbb{C}_J)}}\int_{\partial\Omega} G(\omega,T)\mathbf{D}\omega F_L(s,\omega) ds_J\right)\\
			&&= -\frac 14\left( \int_{\partial \Omega} G(\omega,T) \mathbf{D}\omega \left(\int_{\partial{(G \cap \mathbb{C}_J)}} F_L(s,\omega) ds_J s \right)-T\int_{\partial \Omega} G(\omega,T) \mathbf{D}\omega \left(\int_{\partial{(G \cap \mathbb{C}_J)}} F_L(s,\omega)  ds_J\right)\right)\\
			&&= 0
		\end{eqnarray*}
		where the second equality is a consequence of the Fubini's Theorem and the last equality is a consequence of formulas \eqref{nova} and \eqref{nova_bis}.
	\end{proof}
Finally in the following result we give an answer to the question in Problem \ref{three}.
	\begin{proposition}
		Let $T \in \mathcal{BC}(X)$ be such that $T= T_1e_1+T_2 e_2+T_3 e_3$, and assume that the operators $T_{\ell}$, $\ell=1,2,3$, have real spectrum. Let $U$ be a slice Cauchy domain with $\sigma_S(T)\subset U$. If $f,g\in SH_L(U)$ (resp. $f,g\in SH_R(U)$) satisfy the property $\mathcal{D}f=\mathcal{D}g$ (resp $f\mathcal{D}=g\mathcal{D}$) then $\tilde f(T)=\tilde g(T)$.
	\end{proposition}
	\begin{proof}
		If $U$ is connected we can use Proposition \ref{con}. If $U$ is not connected then $U=\cup_{l=1}^n U_l$ where the $U_l$ are the connected components of $U$. Hence, we have $f(s)-g(s)=\sum_{l=1}^n\chi_{U_l}(s)\alpha_l$ and we can write
		$$
		\tilde f(T)-\tilde g(T)=-\sum_{l=1}^n\frac 1{\pi}\int_{\partial(U_l\cap\mathbb C_J)}\mathcal{Q}_{c,s}(T)^{-1}ds_J\alpha_l.
		$$
		The last summation is zero by Lemma \ref{mono}.
	\end{proof}
	
	We conclude this section with some algebraic properties of the harmonic functional calculus.
	\begin{proposition}
		Let $T\in\mathcal{BC}(X)$ be such that $T=T_1e_1+T_2e_2+T_3e_3$, and assume that the operators $T_\ell$, $\ell=1,\,2,\,3$, have real spectrum.
		\begin{itemize}
			\item If $\tilde f=\mathcal{D}f$ and $\tilde g=\mathcal{D}g$ with $f,g\in SH_L(\sigma_S(T))$ and $a\in\hh$, then
			$$ (\tilde fa+\tilde g)(T)=\tilde f (T)a+\tilde g(T). $$
			\item If $\tilde f=f\mathcal{D}$ and $\tilde g=g\mathcal{D}$ with $f,g\in SH_R(\sigma_S(T))$ and $a\in\hh$, then
			$$ (a\tilde f+\tilde g)(T)=a\tilde f (T)+\tilde g(T). $$
		\end{itemize}
	\end{proposition}
	\begin{proof}
		The obove identities follow immediately from the linearity of the integrals in \eqref{inte1}, resp. \eqref{inte2}.
	\end{proof}
	
	\begin{proposition}
		Let $T\in\mathcal{BC}(X)$ be such that $T=T_1e_1+T_2e_2+T_3e_3$, and assume that the operators $T_\ell$, $\ell=1,\,2,\,3$, have real spectrum.
		\begin{itemize}
			\item If $\tilde f=\mathcal{D}f$ with $f\in SH_L(\sigma_S(T))$ and assume that $f(q)=\sum_{m=0}^{\infty} q^ma_m$ with $a_m\in\mathbb H$, where this series converges on a ball $B_r(0)$ with $\sigma_S(T)\subset B_r(0)$. Then
			$$ \tilde f(T)=-2\sum_{m=1}^\infty  \sum_{k=1}^m T^{m-k} \overline T^{k-1} a_m.$$
			\item If $\tilde f=f\mathcal{D}$ with $f\in SH_R(\sigma_S(T))$ and assume that $f(q)=\sum_{m=0}^{\infty} a_mq^m$ with $a_m\in\mathbb H$, where this series converges on a ball $B_r(0)$ with $\sigma_S(T)\subset B_r(0)$. Then
			$$ \tilde f(T)=-2\sum_{m=1}^\infty  \sum_{k=1}^m a_m T^{m-k} \overline T^{k-1}.$$
		\end{itemize}
	\end{proposition}
	\begin{proof}
		We prove the first assertion since the second one can be proven similarly. We choose an imaginary unit $J\in\mathbb S$ and a radius $0<R<r$ such that $\sigma_S(T)\subset B_R(0)$. Then the series expansion of $f$ converges uniformly on $\partial (B_R(0)\cap\cc_J)$, and so
		$$
		\tilde f(T)=-\frac 1{\pi} \int_{\partial (B_R(0)\cap\cc_J)} \mathcal Q_{c,s}(T)^{-1}\,ds_J\, \sum_{l=0}^\infty s^la_l=-\frac 1{\pi}\sum_{l=0}^\infty \int_{\partial (B_R(0)\cap\cc_J)} \mathcal Q_{c,s}(T)^{-1}\, ds_J s^la_l.
		$$
		By replacing $\mathcal Q_{c,s}(T)^{-1}$ with its series expansion, see Proposition \ref{p2}, we further obtain
		\[
		\begin{split}
			\tilde f(T)&=-\frac 1{\pi} \int_{\partial (B_R(0)\cap\cc_J)} \sum_{m=1}^\infty  \sum_{k=1}^m T^{m-k} \overline T^{k-1} s^{-1-m}\,ds_J\, \sum_{l=0}^\infty s^la_l\\
			& =-\frac 1{\pi} \sum_{m=1}^\infty  \sum_{k=1}^m \sum_{l=0}^\infty T^{m-k} \overline T^{k-1} \int_{\partial (B_R(0)\cap\cc_J)} s^{-1-m}\, ds_J\, s^la_l\\
			&= -2\sum_{m=1}^\infty  \sum_{k=1}^m T^{m-k} \overline T^{k-1} a_m.
		\end{split}
		\]
		The last equality is due to the fact that $\int_{\partial (B_R(0)\cap\cc_J)}s^{-1-m}\, ds_J\, s^l$ is equal to $2\pi$ if $l=m$, and $0$ otherwise.
	\end{proof}
	
\section{The resolvent equations for  harmonic functional calculus}

In this section we prove various resolvent equations for the
pseudo $S$-resolvent operator $\mathcal{Q}_{c, s}(T)^{-1}$.
The first version of this equation is written in terms of $\mathcal{Q}_{c, s}(T)^{-1}$ and of the $S$-resolvent operators.

\begin{theorem}[The $Q$-resolvent equation with $S$-resolvent operators]
	\label{pseudo}
	Let $T \in \mathcal{BC}(X)$. Then, for $p$, $s \in \rho_{S}(T)$ with $s\not\in [p]$, the following equalities hold
	\begin{eqnarray}
		\label{star5}
		\mathcal{Q}_{c, s}(T)^{-1} \mathcal{Q}_{c,p}(T)^{-1} &=& \bigl \{ \left[ \mathcal{Q}_{c,s}(T)^{-1} S_{L}^{-1}(p,T)-S_{R}^{-1}(s,T) \mathcal{Q}_{c,p}(T)^{-1}\right]p\\
		\nonumber
		&& - \bar{s}\left[ \mathcal{Q}_{c,s}(T)^{-1} S_{L}^{-1}(p,T)-S_{R}^{-1}(s,T) \mathcal{Q}_{c,p}(T)^{-1}\right] \bigl \} (p^{2}-2s_0p+|s|^2)^{-1},
	\end{eqnarray}
	and
	\begin{eqnarray}
		\label{star6}
		\! \! \! \!
		\mathcal{Q}_{c,s}(T)^{-1} \mathcal{Q}_{c,p}(T)^{-1}  \! \! \! &=& \! \! \! \,(p^{2}-2s_0p+|s|^2)^{-1} \bigl \{ s\left [ \mathcal{Q}_{c,s}(T)^{-1} S_{L}^{-1}(p,T)-S_{R}^{-1}(s,T) \mathcal{Q}_{c,p}(T)^{-1}\right]\\
		\nonumber
		&& - \left[ \mathcal{Q}_{c,s}(T)^{-1} S_{L}^{-1}(p,T)-S_{R}^{-1}(s,T) \mathcal{Q}_{c,p}(T)^{-1}\right]\bar{p} \bigl \} .
	\end{eqnarray}
\end{theorem}
\begin{proof}
	By the definition of left $S$-resolvent operator we have
	\begin{equation}
		\label{first}
		\mathcal{Q}_{c,p}(T)^{-1}p= \overline{T} \mathcal{Q}_{c,p}(T)^{-1}+S_{L}^{-1}(p,T).
	\end{equation}
	By iterating \eqref{first} we get
	\begin{eqnarray*}
		&& \mathcal{Q}_{c,s}(T)^{-1} \mathcal{Q}_{c,p}(T)^{-1}(p^{2}-2s_0p+|s|^2)\\
		&&= \mathcal{Q}_{c,s}(T)^{-1} [ \mathcal{Q}_{c,p}(T)^{-1}p]p-2s_0 \mathcal{Q}_{c,s}(T)^{-1} \mathcal{Q}_{c,p}(T)^{-1}p+ |s|^2 \mathcal{Q}_{c,s}(T)^{-1} \mathcal{Q}_{c,p}(T)^{-1}\\
		&&= \mathcal{Q}_{c,s}(T)^{-1} [ \overline{T} \mathcal{Q}_{c,p}(T)^{-1}+S_{L}^{-1}(p,T)]p-2s_0 \mathcal{Q}_{c,s}(T)^{-1} [ \overline{T} \mathcal{Q}_{c,p}(T)^{-1}+S_{L}^{-1}(p,T)]\\
		&& \, \, \, \,+ |s|^2 \mathcal{Q}_{c,s}(T)^{-1} \mathcal{Q}_{c,p}(T)^{-1}\\
		&&= \mathcal{Q}_{c,s}(T)^{-1} \overline{T} [\mathcal{Q}_{c,p}(T)^{-1}p]+\mathcal{Q}_{c,s}(T)^{-1}S_{L}^{-1}(p,T)p-2s_0 \mathcal{Q}_{c,s}(T)^{-1} [ \overline{T} \mathcal{Q}_{c,p}(T)^{-1}+S_{L}^{-1}(p,T)]\\
		&& \, \, \, \,+ |s|^2 \mathcal{Q}_{c,s}(T)^{-1} \mathcal{Q}_{c,p}(T)^{-1}\\
		&&= \mathcal{Q}_{c,s}(T)^{-1} \overline{T} [\overline{T} \mathcal{Q}_{c,p}(T)^{-1}+S_{L}^{-1}(p,T)]+\mathcal{Q}_{c,s}(T)^{-1}S_{L}^{-1}(p,T)p-2s_0 \mathcal{Q}_{c,s}(T)^{-1} [ \overline{T} \mathcal{Q}_{c,p}(T)^{-1}\\
		&& \, \, \, \,+S_{L}^{-1}(p,T)]+ |s|^2 \mathcal{Q}_{c,s}(T)^{-1} \mathcal{Q}_{c,p}(T)^{-1}.
	\end{eqnarray*}
	Now, by the definition of the right $S$-resolvent operator we have
	\begin{equation}
		\label{second}
		\mathcal{Q}_{c, s}(T)^{-1} \overline{T}=s \mathcal{Q}_{c,s}(T)^{-1}- S_{R}^{-1}(s,T).
	\end{equation}
	This equality implies
	\begin{eqnarray*}
		&& \mathcal{Q}_{c,s}(T)^{-1} \mathcal{Q}_{c,p}(T)^{-1}(p^{2}-2s_0p+|s|^2)\\
		&&= [\mathcal{Q}_{c,s}(T)^{-1} \overline{T}] \overline{T} \mathcal{Q}_{c,p}(T)^{-1}+[\mathcal{Q}_{c,s}(T)^{-1} \overline{T}]S_{L}^{-1}(p,T)+\mathcal{Q}_{c, s}(T)^{-1}S_{L}^{-1}(p,T)p\\
		&& \, \, \, \,-2s_0 [\mathcal{Q}_{c,s}(T)^{-1}  \overline{T}] \mathcal{Q}_{c,p}(T)^{-1}-2s_0 \mathcal{Q}_{c,s}(T)^{-1}S_{L}^{-1}(p,T)+ |s|^2 \mathcal{Q}_{c,s}(T)^{-1} \mathcal{Q}_{c,p}(T)^{-1}\\
		&& = [s \mathcal{Q}_{c,s}(T)^{-1}- S_{R}^{-1}(s,T)] \overline{T} \mathcal{Q}_{c,p}(T)^{-1}+[s \mathcal{Q}_{c,s}(T)^{-1}- S_{R}^{-1}(s,T)]S_{L}^{-1}(p,T)\\
		&& \, \, \, \,+\mathcal{Q}_{c,s}(T)^{-1}S_{L}^{-1}(p,T)p-2s_0 [s \mathcal{Q}_{c,s}(T)^{-1}- S_{R}^{-1}(s,T)] \mathcal{Q}_{c,p}(T)^{-1}-2s_0 \mathcal{Q}_{c,s}(T)^{-1}S_{L}^{-1}(p,T)\\
		&&\, \, \, \, +|s|^2 \mathcal{Q}_{c,s}(T)^{-1} \mathcal{Q}_{c,p}(T)^{-1}\\
		&& =  s[\mathcal{Q}_{c,s}(T)^{-1}\overline{T}]\mathcal{Q}_{c,p}(T)^{-1}- S_{R}^{-1}(s,T) \overline{T} \mathcal{Q}_{c,p}(T)^{-1}+s \mathcal{Q}_{c,s}(T)^{-1}S_{L}^{-1}(p,T)- S_{R}^{-1}(s,T) S_{L}^{-1}(p,T)\\
		&& \, \, \, \,+\mathcal{Q}_{c,s}(T)^{-1}S_{L}^{-1}(p,T)p-2s_0 s \mathcal{Q}_{c,s}(T)^{-1}\mathcal{Q}_{c,p}(T)^{-1}+2 s_0 S_{R}^{-1}(s,T) \mathcal{Q}_{c,p}(T)^{-1}\\
		&& \, \, \, \,-2s_0 \mathcal{Q}_{c,s}(T)^{-1}S_{L}^{-1}(p,T)+ |s|^2 \mathcal{Q}_{c,s}(T)^{-1} \mathcal{Q}_{c,p}(T)^{-1}\\
		&& =  s[s \mathcal{Q}_{c,s}(T)^{-1}- S_{R}^{-1}(s,T)]\mathcal{Q}_{c,p}(T)^{-1}- S_{R}^{-1}(s,T) \overline{T} \mathcal{Q}_{c,p}(T)^{-1}+s \mathcal{Q}_{c,s}(T)^{-1}S_{L}^{-1}(p,T)\\
		&& \, \, \, \, - S_{R}^{-1}(s,T) S_{L}^{-1}(p,T)+\mathcal{Q}_{c,s}(T)^{-1}S_{L}^{-1}(p,T)p-2s_0 s \mathcal{Q}_{c,s}(T)^{-1}\mathcal{Q}_{c,p}(T)^{-1}+2 s_0 S_{R}^{-1}(s,T) \mathcal{Q}_{c,p}(T)^{-1}\\
		&& \, \, \, \,-2s_0 \mathcal{Q}_{c,s}(T)^{-1}S_{L}^{-1}(p,T)+ |s|^2 \mathcal{Q}_{c,s}(T)^{-1} \mathcal{Q}_{c,p}(T)^{-1}.
	\end{eqnarray*}
	Now, since $s^2-2s_0s+|s|^2=0$ we get
	\begin{eqnarray*}
		&& \mathcal{Q}_{c,s}(T)^{-1} \mathcal{Q}_{c,p}(T)^{-1}(p^{2}-2s_0p+|s|^2)\\
		&& = (s^2-2s_0s+|s|^2) \mathcal{Q}_{c,s}(T)^{-1}\mathcal{Q}_{c,p}(T)^{-1}-sS_{R}^{-1}(s,T)\mathcal{Q}_{c,p}(T)^{-1}-S_{R}^{-1}(s,T) \overline{T} \mathcal{Q}_{c,p}(T)^{-1}\\
		&& \, \, \, \,+s \mathcal{Q}_{c,s}(T)^{-1}S_{L}^{-1}(p,T)- S_{R}^{-1}(s,T) S_{L}^{-1}(p,T)+\mathcal{Q}_{c,s}(T)^{-1}S_{L}^{-1}(p,T)p+2 s_0 S_{R}^{-1}(s,T) \mathcal{Q}_{c,p}(T)^{-1}\\
		&& \, \, \, \, -2s_0 \mathcal{Q}_{c,s}(T)^{-1}S_{L}^{-1}(p,T)\\
		&& = -sS_{R}^{-1}(s,T)\mathcal{Q}_{c,p}(T)^{-1}-S_{R}^{-1}(s,T) \overline{T} \mathcal{Q}_{c,p}(T)^{-1}+s \mathcal{Q}_{c,s}(T)^{-1}S_{L}^{-1}(p,T)- S_{R}^{-1}(s,T) S_{L}^{-1}(p,T)\\
		&& \, \, \, \, +\mathcal{Q}_{c,s}(T)^{-1}S_{L}^{-1}(p,T)p+2 s_0 S_{R}^{-1}(s,T) \mathcal{Q}_{c,p}(T)^{-1}-2s_0 \mathcal{Q}_{c,s}(T)^{-1}S_{L}^{-1}(p,T)\\
		&& =-sS_{R}^{-1}(s,T)\mathcal{Q}_{c,p}(T)^{-1}+s \mathcal{Q}_{c,s}(T)^{-1}S_{L}^{-1}(p,T)- S^{-1}_R(s,T)[ \overline{T} \mathcal{Q}_{c,p}(T)^{-1}+S_{L}^{-1}(p,T)]\\
		&& \, \, \, \,+\mathcal{Q}_{c,s}(T)^{-1}S_{L}^{-1}(p,T)p+2 s_0 S_{R}^{-1}(s,T) \mathcal{Q}_{c,p}(T)^{-1}-2s_0 \mathcal{Q}_{c,s}(T)^{-1}S_{L}^{-1}(p,T).
	\end{eqnarray*}
	Finally, by using another time formula \eqref{first} and the fact that $2s_0-s= \bar{s}$ we obtain
	\begin{eqnarray*}
		&& \mathcal{Q}_{c,s}(T)^{-1} \mathcal{Q}_{c,p}(T)^{-1}(p^{2}-2s_0p+|s|^2)\\
		&& =-sS_{R}^{-1}(s,T)\mathcal{Q}_{c,p}(T)^{-1}+s \mathcal{Q}_{c,s}(T)^{-1}S_{L}^{-1}(p,T)- S^{-1}_R(s,T)\mathcal{Q}_{c,p}(T)^{-1}p\\
		&& \, \, \, \,+\mathcal{Q}_{c,s}(T)^{-1}S_{L}^{-1}(p,T)p+2 s_0 S_{R}^{-1}(s,T) \mathcal{Q}_{c,p}(T)^{-1}-2s_0 \mathcal{Q}_{c,s}(T)^{-1}S_{L}^{-1}(p,T)\\
		&& = \left[ \mathcal{Q}_{c,s}(T)^{-1} S_{L}^{-1}(p,T)-S_{R}^{-1}(s,T) \mathcal{Q}_{c,p}(T)^{-1}\right]p- \bar{s}\left[ \mathcal{Q}_{c,s}(T)^{-1} S_{L}^{-1}(p,T)-S_{R}^{-1}(s,T) \mathcal{Q}_{c,p}(T)^{-1}\right].
	\end{eqnarray*}
	It is possible to obtain formula \eqref{star6} with similar computations.
\end{proof}

\begin{remark}
	We can rewrite the equations obtained in Theorem \ref{pseudo} by using the left or right $*$-products, see \cite[Chap. 4]{MR2752913}, in the variables $s,p\in\rho_S(T)$ with $s\not\in [p]$,
	$$
	\mathcal{Q}_{c,s}(T)^{-1} \mathcal{Q}_{c,p}(T)^{-1} =  \left[ \mathcal{Q}_{c,s}(T)^{-1} S_{L}^{-1}(p,T)-S_{R}^{-1}(s,T) \mathcal{Q}_{c,p}(T)^{-1}\right] *_{s, left}(p- \bar{s})(p^{2}-2s_0p+|s|^2)^{-1} \mathcal{I},
	$$
	or
	$$
	\mathcal{Q}_{c,s}(T)^{-1} \mathcal{Q}_{c,p}(T)^{-1} = (p- \bar{s})(p^{2}-2s_0p+|s|^2)^{-1} \mathcal{I} *_{p, right}\left[ \mathcal{Q}_{c,s}(T)^{-1} S_{L}^{-1}(p,T)-S_{R}^{-1}(s,T) \mathcal{Q}_{c,p}(T)^{-1}\right].
	$$
\end{remark}

\begin{theorem}[Left and right generalized $Q$-resolvent equations]
	Let $T\in \mathcal{BC}(X)$ with $s\in \rho_S(T)$ and set
	$$
	\mathcal M_m^L(s,T):=\sum_{i=0}^{m-1} \overline T^iS^{-1}_L(s,T)s^{m-i-1}
	$$
	and
	$$
	\mathcal M_m^R(s,T):=\sum_{i=0}^{m-1} s^{m-i-1} S^{-1}_R(s,T) \overline T^i.
	$$
	Then for $m\geq 1$ and $s\in\rho_S(T)$, the following equations hold
	\begin{equation}
		\label{gen}
		\mathcal Q_{c,s}(T)^{-1}s^m-\overline T^m \mathcal Q_{c,s}(T)^{-1}=\mathcal M_m^L(s,T)
	\end{equation}
	and
	\begin{equation}\nonumber
		s^m \mathcal Q_{c,s}(T)^{-1}- \mathcal Q_{c,s}(T)^{-1} \overline T^m=\mathcal M_m^R(s,T).
	\end{equation}
\end{theorem}
\begin{proof}
	We prove the result by induction on $m$. We will prove only \eqref{gen} since the other equality is proven with similar techniques. The case $m=1$ is trivial because
	$$\mathcal M_1^L(s,t)=S^{-1}_L(s,t)=\mathcal Q_{c,s}^{-1}(T)s-\overline T \mathcal Q_{c,s}^{-1}(T).$$
	We assume that the equation holds for $m-1$ and we will prove it for $m$. By inductive hypothesis, we have
	\[
	\begin{split}
		\overline T^m \qcs &=\overline T\overline T^{m-1} \qcs=\overline T(\qcs s^{m-1}-\mathcal M_{m-1}^L(s,T))\\
		&=\overline T\qcs s^{m-1}-\overline T\mathcal M_{m-1}^L(s,T).
	\end{split}
	\]
	Since
	$$
	\overline T\mathcal M_{m-1}^L(s,T)=\sum_{i=0}^{m-2}\overline T^{i+1}S^{-1}_L(s,T)s^{m-i-2}=\sum_{i=1}^{m-1}\overline T^i S^{-1}(s,T)s^{m-i-1}
	$$
	and
	$$
	\overline T\qcs=\qcs s -S^{-1}_L(s,t),
	$$
	we have
	\begin{equation}\nonumber
		\begin{split}
			\overline T^m \qcs &=\qcs s^{m}-S^{-1}_L(s,T)s^{m-1}-\sum_{i=1}^{m-1}\overline T^i S^{-1}(s,T)s^{m-i-1}\\
			&=\qcs s^m-\sum_{i=0}^{m-1}\overline{T}^i S^{-1}(s,T)s^{m-i-1}=\qcs s^m-\mathcal M_m^L(s,T).
		\end{split}
	\end{equation}
\end{proof}

Now we prove the $Q$-resolvent equation just in terms of the commutative pseudo $S$-resolvent operator.
\begin{theorem}[The $Q$-resolvent equation]\label{Thm64}
	Let $T\in\mathcal{BC}(X)$. Then for $s,\, p\in\rho_S(T)$ with $s\not\in [p]$, we have the following equation
	\begin{eqnarray}\label{resq}
		&s\qcs \qcp p-s\qcs \overline T\qcp\nonumber\\
		& -\qcs \overline T\qcp p+\qcs\overline T^2\qcp\nonumber\\
		&= \left[(s\qcs-p\qcp)p-\overline s(s\qcs-p\qcp)\right](p^2-2s_0p+|s|^2)^{-1}\\
		&+\left[(\overline T\qcp-\qcs\overline T)p-\overline s(\overline T\qcp-\qcs\overline T)\right](p^2-2s_0p+|s|^2)^{-1}.\nonumber
	\end{eqnarray}
\end{theorem}
\begin{proof}
	Starting from the $S$-resolvent equation (see formula \eqref{ress})
	$$
	S^{-1}_R(s,T)S^{-1}_L(p,T)=\left[(S^{-1}_R(s,T)-S^{-1}_L(p,T))p-\overline s(S^{-1}_R(s,T)-S^{-1}_L(p,T))\right](p^2-2s_0p+|s|^2)^{-1}
	$$
	and using the definitions of the $S$-resolvent operators
	$$
	S^{-1}_R(s,T)=\qcs(s\mathcal I-\overline T),
	$$
	and
	$$
	S^{-1}_L(p,T)=(p\mathcal I-\overline T)\qcp,
	$$
	we obtain that the left hand side of the $S$-resolvent equation can be rewritten as
	\begin{equation}
		\label{star1}
		\begin{split}
			\nonumber
			&S^{-1}_R(s,T)S^{-1}_L(p,T)=\qcs (s\mathcal I-\overline T)(p\mathcal I-\overline T)\qcp\\
			&=s\qcs\qcp p-s\qcs \overline T\qcp\\
			\nonumber
			&\,\,\,\,\, -\qcs \overline T\qcp p+\qcs\overline T^2\qcp.
		\end{split}
	\end{equation}
	The right hand side can be rewritten in the following way
	\begin{equation}
		\label{star2}
		\begin{split}
			&\left[(S^{-1}_R(s,T)-S^{-1}_L(p,T))p-\overline s(S^{-1}_R(s,T)-S^{-1}_L(p,T))\right](p^2-2s_0p+|s|^2)^{-1}\\
			\nonumber
			&=\left[ (\qcs (s\mathcal I-\overline T)-(p\mathcal I -\overline T)\qcp)p+\right. \\
			\nonumber
			&\,\,\,\,\,-\left.\overline s (\qcs (s\mathcal I-\overline T)-(p\mathcal I -\overline T)\qcp)\right](p^2-2s_0p+|s|^2)^{-1}\\
			\nonumber
			&= \left[(s\qcs-p\qcp)p-\overline s(s\qcs-p\qcp)\right](p^2-2s_0p+|s|^2)^{-1}\\
			\nonumber
			&\,\,\,\,\,+\left[(\overline T\qcp-\qcs\overline T)p-\overline s(\overline T\qcp-\qcs\overline T)\right] (p^2-2s_0p+|s|^2)^{-1}.
		\end{split}
	\end{equation}
	
	By equating \eqref{star1} and \eqref{star2}, we obtain the assertion.
\end{proof}
\begin{remark}
	It is possible to rewrite \eqref{resq} as
	\begin{equation}\label{equaQ}	\begin{split}
			&s\qcs \qcp p-s\qcs \overline T\qcp +\\
			& -\qcs \overline T\qcp p+\qcs\overline T^2\qcp\\
			&= \left(s\qcs-p\qcp\right)*_{s, left}S^{-1}_L(p,s)+\left(\overline T\qcs-\qcp\overline T\right)*_{s, left}S^{-1}_L(p,s) 
		\end{split}
	\end{equation}
\end{remark}
\begin{remark}
	Formula \eqref{resq} or, equivalently, \eqref{equaQ} can be considered the most appropriate
	$Q$-resolvent equation because
	\begin{itemize}
		\item[(I)] it preserves the left slice hyperholomorphicity in $s$ and the right slice hyperholomorphicity in $p$;\\
		\item[(II)] the product $\qcs\qcp$ (multiplied by monomials or bounded operators) is written in terms of the difference $\qcs-\qcp$ entangled with the left slice hyperholomorphic Cauchy kernel.
	\end{itemize}
	For more information of the properties of the resolvent equations in hyperholomorphic spectral theories see the paper \cite{CDS}.
\end{remark}

\section{The Riesz projectors for  harmonic functional calculus}
We now take advantage of the $Q$-resolvent equation in Theorem \ref{Thm64} to study
the Riesz projectors for the harmonic functional calculus.
In the sequel we need the crucial result originally proved in \cite[Lemma 3.23]{ACGS}.
\begin{lemma}[See \cite{CGKBOOK}]
	\label{app}
	Let $B \in \mathcal{B}(X)$. Let $G$ be an axially symmetric domain and assume $f \in N(G)$. Then for $p \in G$, we have
	$$ \frac{1}{2 \pi} \int_{\partial (G \cap \mathbb{C}_I)} f(s) ds_I (\bar{s}B-Bp)(p^2-2s_0p+|s|^2)^{-1}=Bf(p).$$
\end{lemma}

\begin{theorem}[The Riesz projectors]
	\label{rp}
	Let $T=T_1e_1+T_2e_2+T_3e_3$ and assume that the operators $T_l$, $l=1,\, 2,\, 3$, have real spectrum. Let $\sigma_S(T)=\sigma_1\cup\sigma_2$ with $\operatorname{dist}(\sigma_1,\sigma_2)>0$.
	
	Let $G_1,\, G_2\subset\mathbb H$ be two bounded slice Cauchy domains such that $\sigma_1\subset G_1$, $\overline G_1\subset G_2$ and $\operatorname{dist}(G_2,\sigma_2)>0$. Then the operator
	$$
	\tilde P:=\frac 1{2\pi}\int_{\partial (G_2\cap\mathbb C_J)} s\, ds_J\qcs=\frac 1{2\pi}\int_{\partial (G_1\cap\cc_J)}\qcp \, dp_J p
	$$
	is a projection, i.e.,
	$$
	\tilde P^2=\tilde P.
	$$
	Moreover, the operator $\tilde{P}$ commutes with $T$, i.e. we have
	\begin{equation}
		\label{comm}
		T\tilde{P}=\tilde{P}T.
	\end{equation}
\end{theorem}
\begin{proof}
	First we multiply equation \eqref{resq} by $ds_J$ on the left and we integrate it on $\partial (G_2\cap\mathbb C_J)$ with respect to $ds_J$, and then we multiply it by $dp_J$ on the right and we integrate it on $\partial (G_1\cap\cc_J)$ with respect to $dp_J$. We obtain
	\begin{equation}\begin{split}\label{nova3}
			& \int_{\partial (G_2\cap\cc_J)} s\, ds_J \qcs\, \int_{\partial (G_1\cap\cc_J)} \qcp\, dp_J p\\
			&+\int_{\partial (G_2\cap\cc_J)} s\, ds_J \qcs\, T \int_{\partial (G_1\cap\cc_J)} \qcp\, dp_J\\
			& +\int_{\partial (G_2\cap\cc_J)} ds_J \qcs\, T \int_{\partial (G_1\cap\cc_J)} \qcp\, dp_Jp\\
			&+\int_{\partial (G_2\cap\cc_J)} ds_J \qcs\, T^2 \int_{\partial (G_1\cap\cc_J)} \qcp\, dp_J\\
			&= \int_{\partial(G_2\cap \cc_J)}ds_J\, \int_{\partial (G_1\cap \cc_J)} \left[(s\qcs-p\qcp)p\right.\\
			&\left. -\overline s(s\qcs-p\qcp) \right]\mathcal Q_s(p)^{-1}\, dp_J\\
			&+ \int_{\partial(G_2\cap \mathbb C_J)}ds_J\, \int_{\partial (G_1\cap \mathbb C_J)} \left[(\overline T\qcp-\qcs\overline T)p\right.\\
			&\left.-\overline s(\overline T\qcp-\qcs\overline T) \right]\mathcal Q_s(p)^{-1}\, dp_J,
		\end{split}
	\end{equation}
	where we set $\mathcal Q_s(p)=p^2-2s_0p+|s|^2$. By Lemma \ref{mono} the expression on the left-hand side of \eqref{nova3} simplifies to
	$$
	\int_{\partial (G_2\cap\cc_J)} s\, ds_J \qcs\, \int_{\partial (G_1\cap\cc_J)} \qcp\, dp_J p.
	$$
	Now we focus on the right hand-side of \eqref{nova3}. We start by rewriting it as
	\begin{equation}
		\begin{split}
			\label{star3}
			& \int_{\partial (G_2\cap\cc_J)} ds_J \int _{\partial(G_1\cap\cc_J)}s[\qcs p-\overline s \qcs]\mathcal Q_s(p)^{-1} \, dp_J\\
			& +\int_{\partial (G_2\cap\cc_J)} ds_J \int _{\partial(G_1\cap\cc_J)}[-\qcs \overline Tp+\overline s \qcs \overline T]\mathcal Q_s(p)^{-1}  \, dp_J\\
			& +\int_{\partial (G_2\cap\cc_J)} ds_J \int _{\partial(G_1\cap\cc_J)}[\overline s\qcp p-p \qcp p]\mathcal Q_s(p)^{-1}  \, dp_J\\
			& +\int_{\partial (G_2\cap\cc_J)} ds_J \int _{\partial(G_1\cap\cc_J)}[-\overline s\overline T\qcp +\overline T \qcp p]\mathcal Q_s(p)^{-1}  \, dp_J.
	\end{split}	\end{equation}
	Now, since $\overline G_1\subset G_2$, for any $s\in\partial(G_2\cap\cc_J)$ the functions
	$$p\to p\mathcal Q_s(p)^{-1}$$
	and
	$$ p\to \mathcal Q_s(p)^{-1} $$
	are intrinsic slice hyperholomorphic on $\overline G_1$. By the Cauchy integral formula we have
	$$ \int_{\partial(G_1\cap\cc_J)}p\mathcal Q_s(p)^{-1}\, dp_J=0,\quad \int_{\partial(G_1\cap\cc_J)}\mathcal{Q}_s(p)^{-1}dp_J=0. $$
	Therefore
	$$ \int_{\partial (G_2\cap\cc_J)} ds_J \int _{\partial(G_1\cap\cc_J)}s[\qcs p-\overline s \qcs]\mathcal Q_s(p)^{-1} \, dp_J=0 $$
	and
	$$ \int_{\partial (G_2\cap\cc_J)} ds_J \int _{\partial(G_1\cap\cc_J)}[-\qcs \overline Tp+\overline s \qcs \overline T]\mathcal Q_s(p)^{-1}  \, dp_J=0. $$
	Thus the right hand side in \eqref{star3} is just
	\[
	\begin{split}
		&\int_{\partial (G_2\cap\cc_J)} ds_J \int _{\partial(G_1\cap\cc_J)}[\overline s\qcp p-p \qcp p]\mathcal Q_s(p)^{-1}  \, dp_J\\
		& -\int_{\partial (G_2\cap\cc_J)} ds_J \int _{\partial(G_1\cap\cc_J)}[\overline s\overline T\qcp -\overline T \qcp p]\mathcal Q_s(p)^{-1}  \, dp_J.
	\end{split}
	\]
	We can further simplify the previous expression by applying  Lemma \ref{app} twice: in the first integral for
	$$ B:=p\qcp $$
	and in the second integral for
	$$ B:=\overline{T}\qcp. $$
	Thus we obtain
	\[
	\begin{split}
		&\int_{\partial (G_2\cap\cc_J)} ds_J \int _{\partial(G_1\cap\cc_J)}[\overline s\qcp p-p \qcp p]\mathcal Q_s(p)^{-1}  \, dp_J\\
		& -\int_{\partial (G_2\cap\cc_J)} ds_J \int _{\partial(G_1\cap\cc_J)}[\overline s\overline T\qcs -\overline T \qcs p]\mathcal Q_s(p)^{-1}  \, dp_J\\
		&=2\pi\int_{\partial(G_1\cap\cc_J)}p\qcp\, dp_J-2\pi\int_{\partial (G_1\cap\cc_J)}\overline T\qcp\, dp_J\\
		&=2\pi\int_{\partial(G_1\cap\cc_J)}p\qcp\, dp_J.
	\end{split}
	\]
	In the last equation we have used Lemma \ref{mono}. In conclusion equation \eqref{nova3} reduces to
	$$ \int_{\partial(G_2\cap\cc_J)} s\,ds_J \qcs\int_{\partial(G_1\cap\cc_J)} \qcp\, dp_J p=2\pi\int_{\partial(G_1\cap\cc_J)}\qcp\, dp_Jp$$
	and, by the definition of the operator $\tilde P$ given in the statement, the previous equality means
	$$ \tilde P^2=\tilde P. $$

	Now, we prove \eqref{comm}. Since $T_0=0$ and
	$$ \overline{T} \mathcal{Q}_{c,p}(T)^{-1}= \mathcal{Q}_{c,p}(T)^{-1}p-S_{L}^{-1}(p,T),$$
	we get
	\begin{eqnarray*}
		T \tilde{P}&=&- \frac{1}{2 \pi} \int_{\partial (G_1 \cap \mathbb{C}_J)}\overline{T} \mathcal{Q}_{c,p}(T)^{-1} dp_J p=- \frac{1}{2 \pi} \int_{\partial (G_1 \cap \mathbb{C}_J)} (\mathcal{Q}_{p}(T)^{-1}p-S^{-1}_{L}(p,T)) dp_J p\\
		&=& - \frac{1}{2 \pi} \int_{\partial (G_1 \cap \mathbb{C}_J)} \mathcal{Q}_{c,p}(T)^{-1}dp_jp^2  + \frac{1}{2 \pi}\int_{\partial (G_1 \cap \mathbb{C}_J)}S_{L}^{-1}(p,T) dp_J p.
	\end{eqnarray*}
	Thus we get
	\begin{equation}
		\label{comm1}
		T \tilde{P}=- \frac{1}{2 \pi} \int_{\partial (G_1 \cap \mathbb{C}_J)} \mathcal{Q}_{c,p}(T)^{-1}dp_Jp^2+\frac{1}{2 \pi}\int_{\partial (G_1 \cap \mathbb{C}_J)}S_{L}^{-1}(p,T) dp_J p.
	\end{equation}
	On the other side, since
	$$  \mathcal{Q}_{c,p}(T)^{-1}\overline{T}= \mathcal{Q}_{c,p}(T)^{-1}p-S_{R}^{-1}(p,T),$$
	we get
	\begin{eqnarray*}
		\tilde{P}T&=&- \frac{1}{2 \pi} \int_{\partial(G_1 \cap \mathbb{C}_J)} p dp_J \mathcal{Q}_{c,p}(T)^{-1}\overline{T}\\
		&=&  - \frac{1}{2 \pi} \int_{\partial (G_1 \cap \mathbb{C}_J)} \mathcal{Q}_{c,p}(T)^{-1}dp_Jp^2  + \frac{1}{2 \pi}\int_{\partial (G_1 \cap \mathbb{C}_J)} p dp_J S_{R}^{-1}(p,T).
	\end{eqnarray*}
	From the fact that $p \chi_{G_1}(p)$ is intrinsic slice hyperholomorphic in $G_1$, it follows by \cite[Thm. 3.2.11]{CGKBOOK} that
	\begin{equation}
		\label{comm2}
		\tilde{P}T=- \frac{1}{2 \pi} \int_{\partial (G_1 \cap \mathbb{C}_J)} \mathcal{Q}_{c,p}(T)^{-1}dp_Jp^2  + \frac{1}{2 \pi}\int_{\partial (G_1 \cap \mathbb{C}_J)}S_{L}^{-1}(p,T) p  dp_J.
	\end{equation}
	Since \eqref{comm1} and \eqref{comm2} are equal we get the statement.
\end{proof}

\begin{remark}
	The $Q$-resolvent equation stated in Theorem \ref{pseudo} preserves the slice hyperholomorphicity, however it is not useful to prove Theorem \ref{rp}.
\end{remark}

\begin{remark}
	Theorem \ref{rp} can be proved using directly the $F$-Functional Calculus. Indeed, in the same hypothesis of the theorem, it is proved in \cite[Thm. 7.4.2]{CGKBOOK} that
	$$
	\check P_1^2=\check P_1\quad\textrm{and}\quad T\check P_1=\check P_1 T
	$$
	for
	$$\check P_1:=-\frac 1{8\pi}\int_{\partial (G_1\cap\cc_J)} F_L(p,T) \, dp_J p^2.
	$$
	Now, using Theorem \ref{fre} and \cite[Lemma 7.4.1]{CGKBOOK}, we have
	\[
	\begin{split}
		&\tilde P=\frac 1{2\pi}\int_{\partial (G_1\cap\cc_J)}\qcp \, dp_J p=-\frac 1{8\pi}\int_{\partial (G_1\cap\cc_J)} F_L(p,T)p-TF_L(p,T) \, dp_J p\\
		&=-\frac 1{8\pi}\int_{\partial (G_1\cap\cc_J)} F_L(p,T) \, dp_J p^2=\check P.
	\end{split}
	\]
\end{remark}

\section{Further properties of the harmonic functional calculus}

In this section we prove other important properties of the harmonic functional calculus.

\begin{theorem}
	\label{poli1}
	Let $T \in \mathcal{BC}(X)$. Let $m \in \mathbb{N}_0$, and let $U \subset \mathbb{H}$ be a bounded slice Cauchy domain with $\sigma_{S}(T) \subset U$. For every $J \in \mathbb{S}$ we have
	\begin{equation}
		\label{bege}
		H_m(T)=\frac{1}{2 \pi} \int_{ \partial(U \cap \mathbb{C}_J)} \mathcal{Q}_{c,s}(T)^{-1} ds_J s^{m+1},
	\end{equation}
	where $$H_m(T):=\sum_{k=0}^{m} T^{m-k} \overline{T}^{k}.$$
\end{theorem}
\begin{proof}
	We start by considering $U$ to be the ball $B_r(0)$ with $\| T \| <r$. We know that
	$$
	\mathcal{Q}_{c,s}(T)^{-1}= \sum_{n=1}^{+ \infty} \sum_{k=1}^n T^{n-k}\overline T^{k-1} s^{-1-n}
	$$
	for every $s \in \partial B_r(0)$. By Proposition \ref{series} we know that the series converges on $ \partial B_{r}(0)$. Thus we have
	\[
	\begin{split}
		\frac{1}{2 \pi} \int_{ \partial(B_r(0) \cap \mathbb{C}_J)} \mathcal{Q}_{c,s}(T)^{-1} ds_J s^{m+1}&=\frac{1}{2 \pi} \sum_{n=1}^{+ \infty} \sum_{k=1}^n T^{n-k}\overline T^{k-1}\int_{ \partial(B_r(0) \cap \mathbb{C}_J)} s^{-n+m} ds_J\\
		&=\sum_{k=1}^{m+1} T^{m+1-k}\overline T^{k-1}=\sum_{k=0}^{m} T^{m-k} \overline T^{k}=H_m(T)
	\end{split}
	\]
	since
	$$ \int_{\partial(B_{r}(0) \cap \mathbb{C}_J)} s^{-n+m}ds_J= \begin{cases}
		0 \quad \hbox{if} \, \, n \neq m+1\\
		2 \pi \quad \hbox{if} \, \, n=m+1.
	\end{cases}
	$$
	This proves the result for the case $U=B_r(0)$. Now we get the result for an arbitrary bounded Cauchy domain $U$ that contains $\sigma_S(T)$. Then there exists a radius $r$ such that $ \overline{U} \subset B_{r}(0)$. The operator $ \mathcal{Q}_{c,s}(T)^{-1}$ is right slice hyperholomorphic and the monomial $s^{m+1}$ is left slice hyperholomorphic on the bounded slice Cauchy domain $B_r(0) \setminus U$. By the Cauchy's integral formula we get
	
	\begin{eqnarray*}
		&&  \frac{1}{2 \pi} \int_{ \partial(B_r(0) \cap \mathbb{C}_J)} \mathcal{Q}_{c,s}(T)^{-1} ds_J s^{m+1}- \frac{1}{2 \pi} \int_{ \partial(U \cap \mathbb{C}_J)} \mathcal{Q}_{c,s}(T)^{-1} ds_J s^{m+1}\\
		&&= \frac{1}{2 \pi} \int_{ \partial\left((B_r(0)\setminus U) \cap \mathbb{C}_J\right)} \mathcal{Q}_{c,s}(T)^{-1} ds_J s^{m+1}=0.
	\end{eqnarray*}
	Finally we have
	$$ \frac{1}{2\pi} \int_{ \partial( U\cap \mathbb{C}_J)} \mathcal{Q}_{c,s}(T)^{-1} ds_J s^{m+1}=\frac{1}{2 \pi} \int_{ \partial(B_r(0) \cap \mathbb{C}_J)} \mathcal{Q}_{c,s}(T)^{-1} ds_J s^{m+1}= H_m(T),$$
	and this concludes the proof.
\end{proof}

\begin{remark}
	Unlike what happens in the $S$-functional calculus (see \cite[Thm. 3.2.2]{CGKBOOK}) we do not have a left slice hyperholomorphic polynomial on the left hand side of equality \eqref{bege}, but we have harmonic polynomials. Another difference with respect to \cite[Thm. 3.2.2]{CGKBOOK} is that in Theorem \ref{poli1} we do not have a difference between right and left part, because by Proposition \ref{series}
	$$
	\sum_{m=1}^\infty  \sum_{k=1}^m T^{m-k} \overline T^{k-1} s^{-1-m}=\sum_{m=1}^\infty  \sum_{k=1}^m s^{-1-m} T^{m-k} \overline T^{k-1}=\mathcal Q_{c,s} (T)^{-1}.
	$$
\end{remark}
For the intrinsic functions we have the following result.
\begin{lemma}
	\label{inte4}
	Let $T \in \mathcal{BC}(X)$. If $ f \in N(\sigma_S(T))$ and $U$ is a bounded slice Cauchy domain such that $ \sigma_S(T) \subset U$ and $ \overline{U} \subset dom(f)$, then we have
	$$
	\tilde{f}(T)= - \frac{1}{\pi} \int_{\partial(U \cap \mathbb{C}_J)} \mathcal{Q}_{c,s}(T)^{-1} ds_J f(s)=- \frac{1}{\pi} \int_{\partial(U \cap \mathbb{C}_J)} f(s) ds_J \mathcal{Q}_{c,s}(T)^{-1}.$$
\end{lemma}
\begin{proof}
	It follows by the definitions of intrinsic functions, of the $Q$-functional calculus and Runge's theorem.
\end{proof}
For the $\mathcal{Q}$- functional calculus it is possible to prove a generalized product rule.
\begin{theorem}
	Let $T \in \mathcal{BC}(X)$ and assume $f \in N(\sigma_{S}(T))$ and $g \in SH_L(\sigma_{S}(T))$ then
	\begin{eqnarray}
		\label{product1}
		2[\mathcal{D} \left((.) fg\right)(T)- \overline{T}\mathcal{D}(fg)(T)]&=&f(T) \mathcal{D} \left((.)g\right)(T)-f(T) \overline{T} \mathcal{D}(g)(T)+\\
		\nonumber
		&&+ \mathcal{D} \left(f(.)\right)(T) g(T)- \mathcal{D}(f)(T) \overline{T} g(T).
	\end{eqnarray}
\end{theorem}
\begin{proof}
	Let $G_1$ and $G_2$ be two bounded slice Cauchy domains such that contain $\sigma_{S}(T)$ and $ \overline{G}_1 \subset G_2$ and $\overline{G}_2 \subset dom(f) \cap dom(g)$. We choose $p \in \partial(G_1 \cap \mathbb{C}_J)$ and $s \in \partial(G_2 \cap \mathbb{C}_J)$. By Definition \ref{Sfun} and Definition \ref{qfun} for $J \in \mathbb{S}$ we have
	
	\begin{eqnarray*}
		&&f(T) \mathcal{D} \left((.)g\right)(T)-f(T) \overline{T} \mathcal{D}(g)(T)+ \mathcal{D} \left(f(.)\right)(T) g(T)- \mathcal{D}(f)(T) \overline{T} g(T)\\
		&&=- \frac{1}{2 \pi} \int_{\partial(G_2 \cap \mathbb{C}_J)} f(s) ds_J S_{R}^{-1}(s,T) \frac{1}{\pi} \int_{\partial(G_1 \cap \mathbb{C}_J)} \mathcal{Q}_{c,p}(T)^{-1}pdp_J g(p)\\
		&& \, \, \, \,\, \, \, \,+\frac{1}{2 \pi} \int_{\partial(G_2 \cap \mathbb{C}_J)} f(s) ds_J S_{R}^{-1}(s,T) \frac{\overline{T}}{\pi} \int_{\partial(G_1 \cap \mathbb{C}_J)} \mathcal{Q}_{c,p}(T)^{-1}dp_J g(p)\\
		&& \, \, \, \,\, \, \, \,- \frac{1}{\pi} \int_{\partial(G_2 \cap \mathbb{C}_J)} \mathcal{Q}_{c,s}(T)^{-1} s ds_J f(s) \frac{1}{2 \pi} \int_{\partial(G_1 \cap \mathbb{C}_J)}S^{-1}_L(p,T) dp_J g(p)\\
		&& \, \, \, \,\, \, \, \,+ \frac{1}{\pi} \int_{\partial(G_2 \cap \mathbb{C}_J)} \mathcal{Q}_{c,s}(T)^{-1}ds_J f(s) \frac{\overline{T}}{2 \pi} \int_{\partial(G_1 \cap \mathbb{C}_J)}S^{-1}_L(p,T) dp_J g(p).
	\end{eqnarray*}
	Since the function $f$ is intrinsic by Lemma \ref{inte4} we get
	\begin{eqnarray*}
		&&f(T)\mathcal{ D} \left((.)g\right)(T)-f(T) \overline{T} \mathcal{D}(g)(T)+ \mathcal{D} \left(f(.)\right)(T) g(T)- \mathcal{D}(f)(T) \overline{T} g(T)\\
		&&=- \frac{1}{2 \pi^2} \int_{\partial(G_2 \cap \mathbb{C}_J)} f(s) ds_J S_{R}^{-1}(s,T)  \int_{\partial(G_1 \cap \mathbb{C}_J)} \mathcal{Q}_{c,p}(T)^{-1}pdp_J g(p)\\
		&& \, \, \, \,\, \, \, \,+\frac{1}{2 \pi^2} \int_{\partial(G_2 \cap \mathbb{C}_J)} f(s) ds_J S_{R}^{-1}(s,T) \overline{T} \int_{\partial(G_1 \cap \mathbb{C}_J)} \mathcal{Q}_{c,p}(T)^{-1}dp_J g(p)\\
		&& \, \, \, \,\, \, \, \,- \frac{1}{2\pi^2} \int_{\partial(G_2 \cap \mathbb{C}_J)} f(s) ds_Js \mathcal{Q}_{c,s}(T)^{-1}  \int_{\partial(G_1 \cap \mathbb{C}_J)}S^{-1}_L(p,T) dp_J g(p)\\
		&& \, \, \, \,\, \, \, \,+ \frac{1}{2\pi^2} \int_{\partial(G_2 \cap \mathbb{C}_J)} f(s) ds_J \mathcal{Q}_{c,s}(T)^{-1} \overline{T} \int_{\partial(G_1 \cap \mathbb{C}_J)}S^{-1}_L(p,T) dp_J g(p)\\
		&&= \frac{1}{2 \pi^2} \int_{\partial(G_2 \cap \mathbb{C}_J)} f(s)ds_J  \int_{\partial (G_1 \cap \mathbb{C}_J)} \biggl[-S_{R}^{-1}(s,T) \mathcal{Q}_{c,p}(T)^{-1}p+S_{R}^{-1}(s,T) \overline{T} \mathcal{Q}_{c,p}(T)^{-1}\\
		&& \, \, \, \,\, \, \, \, -s \mathcal{Q}_{c,s}(T)^{-1} S_{L}^{-1}(p,T)+ \mathcal{Q}_{c,s}(T)^{-1} \overline{T} S_{L}^{-1}(p,T)\biggl] dp_J g(p)\\
		&&= \frac{1}{2 \pi^2} \int_{\partial(G_2 \cap \mathbb{C}_J)} f(s)ds_J  \int_{\partial (G_1 \cap \mathbb{C}_J)} \biggl[-S_{R}^{-1}(s,T) (p \mathcal{I}- \overline{T})\mathcal{Q}_{c,p}(T)^{-1}\\
		&& \, \, \, \,\, \, \, \, + \mathcal{Q}_{c,s}(T)^{-1}(\overline{T}-s \mathcal{I}) S_{L}^{-1}(p,T)\biggl] dp_J g(p)\\
		&&= \frac{1}{2 \pi^2} \int_{\partial(G_2 \cap \mathbb{C}_J)} f(s)ds_J  \int_{\partial (G_1 \cap \mathbb{C}_J)} \biggl[-S_{R}^{-1}(s,T) S^{-1}_{L}(p,T)-S_{R}^{-1}(s,T) S^{-1}_{L}(p,T) \biggl] dp_J g(p)\\
		&&= -\frac{1}{ \pi^2} \int_{\partial(G_2 \cap \mathbb{C}_J)} f(s)ds_J  \int_{\partial (G_1 \cap \mathbb{C}_J)} \biggl[S_{R}^{-1}(s,T) S^{-1}_{L}(p,T)\biggl] dp_J g(p).\\
	\end{eqnarray*}
	By the $S$-resolvent equation (see \eqref{ress}) and by setting $\mathcal{Q}_{s}(p):= p^2-2s_0p+|s|^2$, we get
	\begin{eqnarray*}
		&&f(T) \mathcal{D} \left((.)g\right)(T)-f(T) \overline{T} \mathcal{D}(g)(T)+ \mathcal{D} \left(f(.)\right)(T) g(T)- \mathcal{D}(f)(T) \overline{T} g(T)\\
		&& =  -\frac{1}{ \pi^2} \int_{\partial(G_2 \cap \mathbb{C}_J)} f(s)ds_J  \int_{\partial (G_1 \cap \mathbb{C}_J)} \bigl[(S^{-1}_R(s,T)-S_{L}^{-1}(p,T))p- \bar{s}(S^{-1}_R(s,T)-S_{L}^{-1}(p,T))\bigl] \cdot\\
		&& \, \, \, \,\, \, \, \, \cdot \mathcal{Q}_{s}(p)^{-1} dp_J g(p)\\
		&&=-\frac{1}{ \pi^2}\biggl[ \int_{\partial(G_2 \cap \mathbb{C}_J)} f(s)ds_J  \int_{\partial (G_1 \cap \mathbb{C}_J)} S^{-1}_R(s,T)p\mathcal{Q}_{s}(p)^{-1} dp_J g(p)+\\
		&& \, \, \, \,\, \, \, \,-\int_{\partial(G_2 \cap \mathbb{C}_J)} f(s)ds_J  \int_{\partial (G_1 \cap \mathbb{C}_J)} S^{-1}_L(p,T)p\mathcal{Q}_{s}(p)^{-1} dp_J g(p)\\
		&& \, \, \, \,\, \, \, \,-\int_{\partial(G_2 \cap \mathbb{C}_J)} f(s)ds_J  \int_{\partial (G_1 \cap \mathbb{C}_J)} \bar{s} S^{-1}_R(s,T)\mathcal{Q}_{s}(p)^{-1} dp_J g(p)\\
		&& \, \, \, \,\, \, \, \,+\int_{\partial(G_2 \cap \mathbb{C}_J)} f(s)ds_J  \int_{\partial (G_1 \cap \mathbb{C}_J)} \bar{s} S^{-1}_L(p,T)\mathcal{Q}_{s}(p)^{-1} dp_J g(p) \biggl].
	\end{eqnarray*}
	From the Cauchy formula we obtain
	$$ \int_{\partial(G_2 \cap \mathbb{C}_J)} f(s)ds_J  \int_{\partial (G_1 \cap \mathbb{C}_J)} S^{-1}_R(s,T)p\mathcal{Q}_{s}(p)^{-1} dp_J g(p)=0,$$
	$$ \int_{\partial(G_2 \cap \mathbb{C}_J)} f(s)ds_J  \int_{\partial (G_1 \cap \mathbb{C}_J)} \bar{s}S^{-1}_R(s,T)\mathcal{Q}_{s}(p)^{-1} dp_J g(p)=0.$$
	Therefore
	\begin{eqnarray*}
		&&f(T) \mathcal{D} \left((.)g\right)(T)-f(T) \overline{T} \mathcal{D}(g)(T)+ \mathcal{D} \left(f(.)\right)(T) g(T)- \mathcal{D}(f)(T) \overline{T} g(T)\\
		&&= -\frac{1}{ \pi^2}\biggl[ -\int_{\partial(G_2 \cap \mathbb{C}_J)} f(s)ds_J  \int_{\partial (G_1 \cap \mathbb{C}_J)} S^{-1}_L(p,T)p\mathcal{Q}_{s}(p)^{-1} dp_J g(p)\\
		&& \, \, \, \,\, \, \, \,+\int_{\partial(G_2 \cap \mathbb{C}_J)} f(s)ds_J  \int_{\partial (G_1 \cap \mathbb{C}_J)} \bar{s} S^{-1}_L(p,T)\mathcal{Q}_{s}(p)^{-1} dp_J g(p)\biggl]\\
		&& \, \, \, \, \, \, \, \,=-\frac{1}{ \pi^2} \int_{\partial(G_2 \cap \mathbb{C}_J)} f(s) ds_J \int_{\partial (G_1 \cap \mathbb{C}_J)} \left[\bar{s}S_{L}^{-1}(p,T)-S_{L}^{-1}(p,T)p\right] \mathcal{Q}_{s}(p)^{-1} dp_J g(p).
	\end{eqnarray*}
	Using Lemma \ref{app} with $B:= S^{-1}_L(p,T)$ we get
	\begin{eqnarray*}
		&&f(T) \mathcal{D} \left((.)g\right)(T)-f(T) \overline{T} \mathcal{D}(g)(T)+ \mathcal{D} \left(f(.)\right)(T) g(T)- \mathcal{D}(f)(T) \overline{T} g(T)\\
		&& = - \frac{2}{\pi} \int_{\partial(G_1 \cap \mathbb{C}_J)} S^{-1}_L(p,T) dp_J f(p)g(p).
	\end{eqnarray*}
	Finally by definition of the $S$-resolvent operator we obtain
	\begin{eqnarray*}
		&&f(T) \mathcal{D} \left((.)g\right)(T)-f(T) \overline{T} \mathcal{D}(g)(T)+ \mathcal{D} \left(f(.)\right)(T) g(T)- \mathcal{D}(f)(T) \overline{T} g(T)\\
		&& = - \frac{2}{\pi} \int_{\partial(G_1 \cap \mathbb{C}_J)} (p \mathcal{I}- \overline{T}) \mathcal{Q}_{c,p}(T)^{-1} dp_J f(p)g(p)\\
		&&=- \frac{2}{ \pi} \left(\int_{\partial(G_1 \cap \mathbb{C}_J)}  \mathcal{Q}_{c,p}(T)^{-1} dp_J p f(p)g(p)-\overline{T}\int_{\partial(G_1 \cap \mathbb{C}_J)}   \mathcal{Q}_{c,p}(T)^{-1} dp_J f(p)g(p)\right)\\
		&&= -2 \left[ \overline{T} \mathcal{D}(fg)(T)-\mathcal{D}\left( (.) fg \right)(T)\right].
	\end{eqnarray*}
\end{proof}

\section{The $Q$-functional calculus and new properties of the $F$-functional calculus}

The introduction of the $Q$-functional calculus is essential to prove a product formula for the $F$- functional calculus.
Before to go through this, we prove the following property of the $F$-functional calculus.

\begin{theorem}
	\label{poli2}
	Let us consider $T \in \mathcal{BC}(X)$ and $m \in \mathbb{N}_0$. Let $U \subset \mathbb{H}$ be a bounded slice Cauchy domain with $ \sigma_{S}(T)\subset U$. For every imaginary unit $J \in \mathbb{S}$, we have
	\begin{equation}
		\label{malo1}
		\mathcal{Q}_{m}(T, \overline{T})= - \frac{1}{4 \pi (m+1)(m+2)} \int_{\partial (U \cap \mathbb{C}_J)} F_L(s,T) ds_J s^{m+2}
	\end{equation}
	and
	\begin{equation}
		\label{malo2}
		\mathcal{Q}_{m}(T, \overline{T})= - \frac{1}{4 \pi (m+1)(m+2)} \int_{\partial (U \cap \mathbb{C}_J)} s^{m+2} ds_J F_R(s,T),
	\end{equation}
	where
	$$ \mathcal{Q}_{m}(T, \overline{T})= \sum_{j=0}^m \frac{2(m-j+1)}{(m+1)(m+2)}T^{m-j} \overline{T}^j.
	$$
\end{theorem}
\begin{proof}
	We start by considering the set $U$ as the ball $B_r(0)$ with $|| T|| <r$. Then, by \cite[Thm. 3.9]{CDS} we know that $F_L(s,T)=  \sum_{n=2}^\infty \mathcal -2(n-1)n \mathcal{Q}_{n-2}(T, \overline{T}) s^{-1-n}$ for every $s \in \partial B_{r}(0)$. This series converges uniformly on $\partial B_{r}(0)$. Thus we have
	\[
	\begin{split}
		&-\frac{1}{4 \pi (m+1)(m+2)} \int_{\partial (B_{r}(0) \cap \mathbb{C}_J)} F_L(s,T) ds_J s^{m+2}
		\\
		&= \frac{1}{2\pi (m+1)(m+2)} \sum_{n=0}^{+ \infty} (n+1)(n+2) \mathcal{Q}_{n}(T, \overline{T})  \int_{\partial (B_{r}(0) \cap \mathbb{C}_J)} s^{-1-n+m} ds_J.
	\end{split}
	\]
	Due to the fact that
	$$ \int_{\partial(B_{r}(0) \cap \mathbb{C}_J)} s^{-1-n+m}ds_J= \begin{cases}
		0 \quad  \, \, \hbox{if} \, \, \, \, n \neq m,\\
		2 \pi \quad \hbox{if} \, \, n=m,
	\end{cases}
	$$
	we obtain
	$$ -\frac{1}{4 \pi (m+1)(m+2)} \int_{\partial (B_{r}(0) \cap \mathbb{C}_J)} F_L(s,T) ds_J s^{m+2}=\mathcal{Q}_{m}(T, \overline{T}).$$
	Now, we consider $U$ an arbitrary bounded slice Cauchy domain that contains $ \sigma_{S}(T)$. We suppose that there exists a radius $r$ such that $ \overline{U} \subset B_{r}(0)$. The left $F$-resolvent operator $F_{L}(s,T)$ is right slice hyperholomorphic in the variable $s$ and the monomial $s^{m+2}$ is left slice hyperholomorphic on the bounded slice Cauchy domain $B_{r}(0) \setminus U$. By the Cauchy's integral theorem (see Theorem \ref{CIT}) we get
	\[
	\begin{split}
		& -\frac{1}{4 \pi (m+1)(m+2)} \int_{\partial (B_{r}(0) \cap \mathbb{C}_J)} F_L(s,T) ds_J s^{m+2}
		\\
		&
		+\frac{1}{4 \pi (m+1)(m+2)} \int_{\partial (U \cap \mathbb{C}_J)} F_L(s,T) ds_J s^{m+2}
		\\
		&= -\frac{1}{4 \pi (m+1)(m+2)} \int_{\partial \left((B_{r}(0)\setminus U) \cap \mathbb{C}_J\right)} F_L(s,T) ds_J s^{m+2}=0.
	\end{split}
	\]
	This implies that
	\[
	\begin{split}
		&-\frac{1}{4 \pi (m+1)(m+2)} \int_{\partial (U \cap \mathbb{C}_J)} F_L(s,T) ds_J s^{m+2}
		\\
		&=-\frac{1}{4 \pi (m+1)(m+2)} \int_{\partial (B_{r}(0) \cap \mathbb{C}_J)} F_L(s,T) ds_J s^{m+2}
		\\
		& = \mathcal{Q}_{m}(T, \overline{T}).
	\end{split}
	\]
\end{proof}
\begin{remark}
	Theorem \ref{poli2} is meaningful with respect to the definition of the $F$-functional calculus. Indeed, the results of the integrals \eqref{malo1} and \eqref{malo2} are Fueter regular polynomials in T. That kind of polynomials were introduced in \cite{CMF, CMF1}.
\end{remark}
Now, we prove a product rule for the $F$-function calculus.

\begin{theorem}
	\label{product}
	Let $T \in \mathcal{BC}(X)$ and assume $f \in N(\sigma_S(T))$ and $g \in SH_L(\sigma_S(T))$ then we have
	\begin{equation}
		\label{pr1}
		\Delta (fg)(T)=\Delta f(T) g(T)+f(T) \Delta g(T)-  \mathcal{D}f(T) \mathcal{D} g(T).
	\end{equation}
	where $\mathcal{D}$ is the Fueter operator.
\end{theorem}
\begin{proof}
	Let $G_1$ and $G_2$ be two bounded slice Cauchy domains such that contain $\sigma_{S}(T)$ and $ \overline{G}_1 \subset G_2$, with $\overline{G}_2 \subset dom(f) \cap dom(g)$. We choose $p \in \partial(G_1 \cap \mathbb{C}_J)$ and $s \in \partial(G_2 \cap \mathbb{C}_J)$. For every $J \in \mathbb{S}$, from the definitions of $F$-functional calculus, $S$-functional calculus and $Q$- functional calculus, we get
	\begin{eqnarray*}
		&& \Delta f(T) g(T)+f(T) \Delta g(T)-  \mathcal{D}f(T) \mathcal{D} g(T)= \\
		&&=\frac{1}{(2 \pi)^2} \int_{\partial(G_2 \cap \mathbb{C}_J)} f(s) ds_J F_{R}(s,T) \int_{\partial(G_1 \cap \mathbb{C}_J)} S^{-1}_L (p,T) dp_J g(p)+\\
		&& + \frac{1}{(2 \pi)^2} \int_{\partial(G_2 \cap \mathbb{C}_J)} f(s) ds_J S^{-1}_{R}(s,T) \int_{\partial(G_1 \cap \mathbb{C}_J)} F_L (p,T) dp_J g(p)+\\
		&&- \frac{1}{(\pi)^2} \int_{\partial(G_2 \cap \mathbb{C}_J)} \mathcal{Q}_{c,s}(T)^{-1} ds_J f(s) \int_{\partial(G_1 \cap \mathbb{C}_J)} \mathcal{Q}_{c,p}(T)^{-1} dp_J g(p).
	\end{eqnarray*}
	Since the function $f$ is intrinsic by Lemma \ref{inte4} we have
	\begin{eqnarray*}
		&& \Delta f(T) g(T)+f(T) \Delta g(T)-  \mathcal{D}f(T) \mathcal{D} g(T)= \\
		&&=\frac{1}{(2 \pi)^2} \int_{\partial(G_2 \cap \mathbb{C}_J)} f(s) ds_J F_{R}(s,T) \int_{\partial(G_1 \cap \mathbb{C}_J)} S^{-1}_L (p,T) dp_J g(p)+\\
		&& + \frac{1}{(2 \pi)^2} \int_{\partial(G_2 \cap \mathbb{C}_J)} f(s) ds_J S^{-1}_{R}(s,T) \int_{\partial(G_1 \cap \mathbb{C}_J)} F_L (p,T) dp_J g(p)+\\
		&&- \frac{1}{(\pi)^2} \int_{\partial(G_2 \cap \mathbb{C}_J)} f(s) ds_J \mathcal{Q}_{c,s}(T)^{-1} \int_{\partial(G_1 \cap \mathbb{C}_J)} \mathcal{Q}_{c,p}(T)^{-1} dp_J g(p).
	\end{eqnarray*}
	Hence, we have
	\[
	\begin{split}
		\Delta f(T) g(T)&+f(T) \Delta g(T)- \mathcal{D}f(T) \mathcal{D} g(T)
		\\
		&
		= \frac{1}{(2 \pi)^2} \int_{\partial(G_2 \cap \mathbb{C}_J)} \int_{\partial(G_1 \cap \mathbb{C}_J)} f(s) ds_J \bigl[F_{R}(s,T)S^{-1}_{L}(p,T)+
		\\
		&
		+S^{-1}_{R}(s,T) F_{L}(p,T) -4 \mathcal Q_{c,s}(T)^{-1} \mathcal{Q}_{c,p}(T)^{-1} \bigl] dp_J g(p).
	\end{split}
	\]
	By the following equation (see \cite[Lemma 7.3.2]{CGKBOOK})
	\begin{equation}\nonumber
		\begin{split}
			F_R(s,T) & S^{-1}_L(p,T)+S^{-1}_R(s,T)F_L(p,T)-4\mathcal Q_{c,s}(T)^{-1}\mathcal Q_{c,p}(T)^{-1}\\
			&=[(F_R(s,T)-F_L(p,T))p-\overline s(F_R(s,T)-F_L(p,T))]\mathcal Q_s(p)^{-1},
		\end{split}
	\end{equation}
	where $\mathcal{Q}_s(p)=p^2-2s_0p+|s|^2$, we obtain
	\[
	\begin{split}
		&\Delta f(T) g(T)+f(T) \Delta g(T)-  \mathcal{D}f(T)\mathcal{ D} g(T)
		\\
		&= \frac{1}{(2 \pi)^2} \int_{\partial(G_2 \cap \mathbb{C}_J)} \int_{\partial(G_1 \cap \mathbb{C}_J)} f(s) \bigl[\left(F_{R}(s,T)-F_{L}(p,T)\right)p -\bar{s}\left(F_{R}(s,T)-F_{L}(p,T)\right)\bigl] \mathcal{Q}_s(p)^{-1}dp_J g(p).
	\end{split}
	\]
	By the linearity of the integrals follows that
	\begin{eqnarray*}
		&& \Delta f(T) g(T)+f(T) \Delta g(T)-  \mathcal{D}f(T) \mathcal{D} g(T)= \\
		&&= \frac{1}{(2\pi)^2} \int_{\partial(G_2 \cap \mathbb{C}_J)} f(s) ds_J \int_{\partial(G_1 \cap \mathbb{C}_J)} F_R(s,T) p \mathcal{Q}_s(p)^{-1} dp_J g(p)+\\
		&& -\frac{1}{(2\pi)^2} \int_{\partial(G_2 \cap \mathbb{C}_J)} f(s) ds_J \int_{\partial(G_1 \cap \mathbb{C}_J)} F_L(p,T) p \mathcal{Q}_s(p)^{-1} dp_J g(p)+\\
		&& -\frac{1}{(2\pi)^2} \int_{\partial(G_2 \cap \mathbb{C}_J)} f(s) ds_J \int_{\partial(G_1 \cap \mathbb{C}_J)} \bar{s} F_R(s,T) \mathcal{Q}_s(p)^{-1} dp_J g(p)+\\
		&& +\frac{1}{(2\pi)^2} \int_{\partial(G_2 \cap \mathbb{C}_J)} f(s) ds_J \int_{\partial(G_1 \cap \mathbb{C}_J)} \bar{s} F_L(p,T) \mathcal{Q}_s(p)^{-1} dp_J g(p).
	\end{eqnarray*}
	Since the functions $p \mapsto p \mathcal{Q}_s(p)^{-1}$, $p \mapsto \mathcal{Q}_s(p)^{-1}$ are intrinsic slice hyperholomorphic on $\bar{G}_1$, by the Cauchy integral formula, see Theorem \ref{CIT} we have
	$$ \frac{1}{(2\pi)^2} \int_{\partial(G_2 \cap \mathbb{C}_J)} f(s) ds_J \int_{\partial(G_1 \cap \mathbb{C}_J)} F_R(s,T) p \mathcal{Q}_s(p)^{-1} dp_J g(p)=0,$$
	$$ \frac{1}{(2\pi)^2} \int_{\partial(G_2 \cap \mathbb{C}_J)} f(s) ds_J \int_{\partial(G_1 \cap \mathbb{C}_J)} \bar{s}F_R(s,T) \mathcal{Q}_s(p)^{-1} dp_J g(p)=0.$$
	Thus, we get
	\begin{eqnarray*}
		&& \Delta f(T) g(T)+f(T) \Delta g(T)-  \mathcal{D}f(T) \mathcal{D} g(T)= \\
		&& -\frac{1}{(2\pi)^2} \int_{\partial(G_2 \cap \mathbb{C}_J)} f(s) ds_J \int_{\partial(G_1 \cap \mathbb{C}_J)} F_L(p,T) p \mathcal{Q}_s(p)^{-1} dp_J g(p)+\\
		&& +\frac{1}{(2\pi)^2} \int_{\partial(G_2 \cap \mathbb{C}_J)} f(s) ds_J \int_{\partial(G_1 \cap \mathbb{C}_J)} \bar{s} F_L(p,T) \mathcal{Q}_s(p)^{-1} dp_J g(p)\\
		&&= \frac{1}{(2 \pi)^2} \int_{\partial (G_2 \cap \mathbb{C}_J)} f(s) ds_J \int_{\partial (G_1 \cap \mathbb{C}_J)} \bigl[\bar{s} F_L(p,T)- F_L(p,T)p \bigl] \mathcal{Q}_s(p)^{-1} dp_J g(p).
	\end{eqnarray*}
	By applying Lemma \ref{app} with $B:= F_L(p,T)$ and by the definition of the $F$-functional calculus we obtain
	
	\begin{eqnarray*}
		&& \Delta f(T) g(T)+f(T) \Delta g(T)-  \mathcal{D}f(T)\mathcal{D} g(T)=  \frac{1}{2 \pi} \int_{\partial(G_1 \cap \mathbb{C}_J)} F_L(p,T) dp_J f(p)g(p)\\
		&&= \frac{1}{2 \pi} \int_{\partial(G_1 \cap \mathbb{C}_J)} F_L(p,T) dp_J (fg)(p)=\Delta (fg)(T).
	\end{eqnarray*}
	
\end{proof}

\begin{corollary}
	\label{product2}
	Let $T \in \mathcal{BC}(X)$ and assume $g \in N(\sigma_S(T))$ and $f\in SH_R(\sigma_S(T))$ then we have
	\begin{equation}
		\label{star4}
		\Delta (fg)(T)=\Delta f(T) g(T)+f(T) \Delta g(T)-  \mathcal{D}f(T)\mathcal{ D} g(T).
	\end{equation}
\end{corollary}

\begin{remark}
	The product $fg$ in Theorem \ref{product} and Corollary \ref{product2} is respectively slice hyperholomorphic left or right slice hyperholomorphic.
\end{remark}

\begin{remark}
	Formula \eqref{pr1} is a general case of the well-known formula $\Delta(qg(q))=q \Delta (g(q))+2 \mathcal{D}(g(q))$. Indeed, it is enough to change the operator $T$ with $q$ and to take $f(q):=q$ in formula \eqref{pr1}.
\end{remark}


\begin{thebibliography}{99}
	
	\bibitem{COF}
	D.Alpay, F.Colombo, I.Sabadini,
	{\em Quaternionic de Branges spaces and characteristic operator function},
	SpringerBriefs in Mathematics, Springer, Cham, 2020.
	
	\bibitem{ack}
	D.Alpay, F.Colombo,  D.P.Kimsey,
	{\em  The spectral theorem for quaternionic unbounded normal operators based on the $S$-spectrum},
	J. Math. Phys., {\bf 57}(2016), pp. 023503, 27.
	
	\bibitem{ACGS}  D.Alpay, F.Colombo, J.Gantner, I.Sabadini, {\em A new resolvent equation for the S-functional calculus}, J.Geom. Anal., {\bf 25} (2015), pp. 1939--1968.
	
	
	
	
	\bibitem{Hinfty}
	D.Alpay, F.Colombo, J.Gantner, I.Sabadini,
	{\em The $H^\infty$functional calculus based on the S-spectrum for quaternionic operators and for n-tuples of noncommuting operators},
	J. Funct. Anal., {\bf 271}(2016), pp. 1544--1584.
	
	
	
	
	
	\bibitem{ACSBOOK}
	D.Alpay, F.Colombo, I.Sabadini,
	{\em Slice Hyperholomorphic Schur Analysis},
	Volume 256 of {\em Operator Theory: Advances and Applications}. Birkh\"{a}user, Basel. 2017.
	
	\bibitem{frac6}
	L.Baracco, F.Colombo, M. M.Peloso, S.Pinton,
	{\em  Fractional powers of higher order vector operators on bounded and unbounded domains},
	Preprint. arXiv:2112.05380.
	
	\bibitem{B} H.Begeher, \emph{Iterated integral opererators in Clifford analysis}, J. Anal. Appl. \textbf{18} (1999), 361-377 .
	
	\bibitem{BF}
	G.Birkhoff, J.Von Neumann, {\em The logic of quantum mechanics}, Ann. of
	Math., {\bf 37} (1936), 823--843.
	
	
	\bibitem{CMF}  I.Cação, M.I.Falcão, H.Malonek, {\em Laguerre derivative and monogenic Laguerre polynomials: An
		operational approach}. Mathematical and Computer Modelling 53. 1084-1094. (2011)
	
	\bibitem{CMF1}  I.Cação, M. I.Falcão,  H.Malonek, {\it Hypercomplex Polynomials, Vietoris' Rational Numbers and a Related Integer Numbers Sequence},
	Complex Anal. Oper. Theory, {\bf 11}, (2017), 1059--1076.
	
	
	\bibitem{FIVEDIM} F.Colombo, A.De Martino, S. Pinton, I. Sabadini, \emph{
		The fine structure of the spectral theory on the $S$-spectrum in dimension five}, preprint 2022.
	
	
	\bibitem{CDS} F.Colombo, A.De Martino, I.Sabadini, \emph{The $\mathcal{F}$-resolvent equation and Riesz projectors for the $\mathcal{F}$-functional calculus }, (submitted) arXiv 2112.04830.
	
	\bibitem{CG} F.Colombo, J.Gantner, \emph{Formulations of the $ \mathcal{F}$- functional calculus and some consequences}, Proceedings of the Royal Society of Edinburgh, \textbf{146 A}  (2016), 509-545.
	
	
	
	
	\bibitem{64FRAC}
	F.Colombo, J.Gantner,
	{\em An application of the $S$-functional calculus to fractional diffusion processes},
	Milan J. Math., {\bf 86} (2018), 225--303.
	
	
	
	\bibitem{FJBOOK}
	F.Colombo, J.Gantner,
	{\em Quaternionic closed operators, fractional powers and fractional diffusion processes},
	Operator Theory: Advances and Applications, 274. Birkh\"auser/Springer, Cham, 2019. viii+322.
	
	\bibitem{CGKBOOK} F.Colombo, J.Gantner, D.P.Kimsey,
	{\em Spectral theory on the $S$-spectrum for quaternionic operators},
	Operator Theory: Advances and Applications, 270.
	Birkh\"auser/Springer, Cham, 2018. ix+356 pp.
	
	
	\bibitem{frac4}
	F.Colombo, D.Deniz-Gonzales, S.Pinton,
	{\em Fractional powers of vector operators with first order boundary conditions},
	J. Geom. Phys., {\bf 151} (2020), 103618, 18 pp.
	
	\bibitem{frac5}
	F.Colombo, D.Deniz-Gonzales, S.Pinton,
	{\em Non commutative fractional Fourier law in bounded and unbounded domains},
	Complex Anal. Oper. Theory, {\bf 15} (2021), no. 7, Paper No. 114, 27 pp.
	
	
	
	\bibitem{SPECT-CLIFF}
	F.Colombo, D.P.Kimsey, {\em The spectral theorem for normal operators on a Clifford module}, Anal. Math. Phys. 12 (2022), no. 1, Paper No. 25.
	
	
	\bibitem{CS}
	F.Colombo, I.Sabadini, {\em The F-functional calculus for unbounded operators},
	J. Geom. Phys., {\bf 86} (2014), 392--407.
	
	\bibitem{CSS1}  F.Colombo, I.Sabadini, F.Sommen,
	{\em The Fueter mapping theorem in integral form and the F-functional calculus}, Math. Methods Appl. Sci., {\bf 33} (2010), 2050--2066.
	
	
	
	
	
	\bibitem{MR2752913}
	F.Colombo, I.Sabadini, D.C.Struppa,
	{\em Noncommutative functional calculus. Theory and applications of slice hyperholomorphic functions},
	Volume 289 of {\em Progress
		in Mathematics}.
	Birkh\"auser/Springer Basel AG, Basel. 2011.
	
	
	\bibitem{JFACSS}
	F. Colombo, I.Sabadini, D.C.Struppa,
	{\em A new functional calculus for non commuting operators},
	J. Funct. Anal., {\bf 254} (2008), 2255-2274.
	
	\bibitem{CSS3} F.Colombo, I.Sabadini, D.C.Struppa, \emph{Michele Sce's Works in Hypercomplex Analysis. A Translation with Commentaries}, Birkhäuser/Springer Basel AG, Basel, 2020.
	
	\bibitem{CS0} A.K.Common, F.Sommen, \emph{Axial monogenic functions from holomorphic functions}, J.
	Math. Anal. Appl., \textbf{179} (1993), 610-629.
	
\bibitem{DP} A.De Martino, S.Pinton, \emph{A polynalytic functional calcuus of order 2 on the $S$-Spectrum}, preprint 2022.
	
	\bibitem{F} R.Fueter, \emph{Die Funktionentheorie der Differentialgleichungen $\Delta u = 0$ und $\Delta\Delta u = 0$ mit vier reellen Variablen},  Comm. Math. Helv., {\bf 7} (1934), 307-330.
	\bibitem{J} B.Jefferies, {\em Spectral properties of noncommuting operators},
	Lecture Notes in Mathematics, 1843, Springer-Verlag, Berlin, 2004.
	\bibitem{JM} B.Jefferies, A.McIntosh, J.Picton-Warlow,  {\em The monogenic functional
		calculus}, Studia Math., {\bf 136} (1999), 99-119.
\bibitem{KP} S.G.Krantz, H.R.Parks \emph{A Primer of Real Analytic Functions} 2nd edition, Birkhäuser/Springer, Boston, 2002.

	\bibitem{Dixan} D.Pena-Pena, \emph{Cauchy Kowalevski extensions, Fueter's theorems and boundary values of special systems in Clifford analysis}, PhD Dissertation, Gent, 2008.
	
	\bibitem{Q} T.Qian, \emph{Generalization of Fueters result to $\mathbb{R}^{n+1}$}, Rend. Mat. Acc. Lincei, \textbf{9} (1997),  111-117.
	
	\bibitem{TAOBOOK}
	T.Qian, P.Li,
	\emph{Singular integrals and Fourier theory on Lipschitz boundaries},
	Science Press Beijing, Beijing; Springer, Singapore, 2019. xv+306 pp.
	
	
	
	\bibitem{S} M.Sce, \emph{Osservazioni sulle serie di potenze nei moduli quadratici}, Atti Accad. Naz. Lincei. Rend. CI. Sci.Fis. Mat. Nat. \textbf{23} (1957), 220-225.
\end{thebibliography}
\end{document}